\documentclass[11pt,a4paper]{amsart}
\usepackage{mathptmx, amssymb,amscd,latexsym, eulervm}
\usepackage[all]{xy}
\usepackage{changepage} 

\usepackage{xcolor}
\definecolor{citepink}{HTML}{AA3377}

\usepackage{amsmath}
\usepackage{amsthm}
\usepackage{mathdots}
\usepackage[onehalfspacing]{setspace}
\usepackage{amsopn}
\usepackage{etaremune, mathtools} 
\usepackage{stmaryrd} 
\usepackage{tabularx}
\usepackage{amsfonts}
\usepackage{paralist}
\usepackage{aliascnt}
\usepackage[initials, lite]{amsrefs}
\usepackage{amscd}
\usepackage{blkarray}
\usepackage{mathbbol}
\usepackage{setspace}
\usepackage[inner=2.4cm,outer=2.4cm, bottom=3.2cm]{geometry}
\usepackage{tikz, tikz-cd}
\usepackage{calligra,mathrsfs}
\usepackage{comment}

\usepackage[
  colorlinks=true,
  citecolor=citepink,
  linkcolor=blue,
  urlcolor=blue
]{hyperref}

\usepackage{tikz}
\usetikzlibrary{matrix}
\usetikzlibrary{arrows,calc}
\usetikzlibrary{decorations.markings,intersections,positioning,calc}

  \tikzset{mylabel/.style  args={at #1 #2  with #3}{
    postaction={decorate,
    decoration={
      markings,
      mark= at position #1
      with  \node [#2] {#3};
 } } } }
\allowdisplaybreaks

\BibSpec{collection.article}{%
	+{}  {\PrintAuthors}                {author}
	+{,} { \textit}                     {title}
	+{.} { }                            {part}
	+{:} { \textit}                     {subtitle}
	+{,} { \PrintContributions}         {contribution}
	+{,} { \PrintConference}            {conference}
	+{}  {\PrintBook}                   {book}
	+{,} { }                            {booktitle}
	+{,} { }                            {series}
	+{, vol.} { }                            {volume}
	+{,} { }                            {publisher}
	+{,} { \PrintDateB}                 {date}
	+{,} { pp.~}                        {pages}
	+{,} { }                            {status}
	+{,} { \PrintDOI}                   {doi}
	+{,} { available at \eprint}        {eprint}
	+{}  { \parenthesize}               {language}
	+{}  { \PrintTranslation}           {translation}
	+{;} { \PrintReprint}               {reprint}
	+{.} { }                            {note}
	+{.} {}                             {transition}
	+{}  {\SentenceSpace \PrintReviews} {review}
}
\AtBeginDocument{%
	\def\MR#1{}
}

\makeatletter
\@namedef{subjclassname@2020}{%
	\textup{2020} Mathematics Subject Classification}
\makeatother

\newcommand{\lk}{\mathrm{lk}}

\newcommand{\tr}{\operatorname{tr}}
\newcommand{\prk}{\operatorname{p}.\operatorname{rk}}
\newcommand{\cprk}{\operatorname{cop}.\operatorname{rk}}
\newcommand{\fprk}{\operatorname{fp}.\operatorname{rk}}
\newcommand{\cfprk}{\operatorname{cofp}.\operatorname{rk}}

\def\opn#1#2{\def#1{\operatorname{#2}}}
\opn\Cl{Cl} \opn\deg{deg} \opn\Stab{Stab} \opn\aff{aff} \opn\div{div}
\opn\cone{cone} \opn\End{End} \opn\mod{mod}  \opn\pdim{pdim} \opn\diag{diag} \opn\vert{vert} \opn\m{m} \opn\V{V} \opn\HT{ht}
\opn\Cone{Cone} \opn\Pyr{Pyr} \opn\max{max} \opn\min{min} \opn\int{int} \opn\rev{rev} \opn\ker{ker} \opn\lat{lat} \opn\pull{pull}
\opn\cok{coker} \opn\ant{ant}
\opn\inte{int} \opn\projdim{pd}
\opn\ch{ch}

\newcommand{\kk}{\mathbb{k}}

\newcommand{\ZZ}{\normalfont\mathbb{Z}}
\newcommand{\setZ}{\normalfont\mathbb{Z}}
\newcommand{\setQ}{\normalfont\mathbb{Q}}

\newcommand{\mm}{{\normalfont\mathfrak{m}}}

\newcommand{\QQ}{\mathbb{Q}}
\newcommand{\pp}{{\normalfont\mathfrak{p}}}

\newcommand{\Ker}{\normalfont\text{Ker}}

\newcommand{\ann}{\normalfont\text{Ann}}

\newcommand{\Ass}{{\normalfont\text{Ass}}}

\newcommand{\Hom}{\normalfont\text{Hom}}

\newcommand{\Spec}{\normalfont\text{Spec}}

\def\f0{\mathbf{0}}

\def\1{\mathbf{1}}

\newtheorem{theorem}{Theorem}[section]

\newaliascnt{headcor}{headthm}

\aliascntresetthe{headcor}

\newaliascnt{headconj}{headthm}

\aliascntresetthe{headconj}

\newaliascnt{corollary}{theorem}
\newtheorem{corollary}[corollary]{Corollary}
\aliascntresetthe{corollary}

\newaliascnt{claim}{theorem}

\aliascntresetthe{claim}

\newaliascnt{lemma}{theorem}
\newtheorem{lemma}[lemma]{Lemma}
\aliascntresetthe{lemma}

\newaliascnt{conjecture}{theorem}

\aliascntresetthe{conjecture}

\newaliascnt{proposition}{theorem}
\newtheorem{proposition}[proposition]{Proposition}
\aliascntresetthe{proposition}

\theoremstyle{definition}
\newaliascnt{definition}{theorem}
\newtheorem{definition}[definition]{Definition}
\aliascntresetthe{definition}

\newaliascnt{notation}{theorem}
\newtheorem{notation}[notation]{Notation}
\aliascntresetthe{notation}

\newaliascnt{condition}{theorem}

\aliascntresetthe{condition}

\newaliascnt{example}{theorem}
\newtheorem{example}[example]{Example}
\aliascntresetthe{example}

\newaliascnt{examples}{theorem}

\aliascntresetthe{examples}

\newaliascnt{remark}{theorem}
\newtheorem{remark}[remark]{Remark}
\aliascntresetthe{remark}

\newaliascnt{question}{theorem}

\aliascntresetthe{question}

\newaliascnt{questions}{theorem}

\aliascntresetthe{questions}

\newaliascnt{problem}{theorem}

\aliascntresetthe{problem}

\newaliascnt{construction}{theorem}

\aliascntresetthe{construction}

\newaliascnt{setup}{theorem}
\newtheorem{setup}[setup]{Setup}
\aliascntresetthe{setup}

\newaliascnt{algorithm}{theorem}

\aliascntresetthe{algorithm}

\newaliascnt{observation}{theorem}

\aliascntresetthe{observation}

\newaliascnt{defprop}{theorem}

\aliascntresetthe{defprop}

\newaliascnt{fact}{theorem}

\aliascntresetthe{fact}

\usepackage{mymacros_common}

\begin{document}

\title[Trace ideals of exterior powers of the module of differentials]{Trace ideals of exterior powers of the module of differentials}

\author[R. Ishizuka]{Ryo Ishizuka}
\address{Department of Mathematics, Institute of Science Tokyo, 2-12-1 Ookayama, Meguro, Tokyo 152-8551}
\email{ishizuka.r.ac@m.titech.ac.jp}

\author[S. Miyashita]{Sora Miyashita}
\address{Department of Pure And Applied Mathematics, Graduate School Of Information Science And Technology, Osaka University, Suita, Osaka 565-0871, Japan}
\email{u804642k@ecs.osaka-u.ac.jp}



\begin{abstract}
For each $i \geq 0$, we study the trace ideal of the $i$-th exterior power of the module of differentials. In characteristic zero, we show that these ideals characterize the polynomial rank of graded rings and the formal power series rank of complete local rings with finite residue-field extension, namely the maximal number of variables for a polynomial or formal power series extension over a subring.
Moreover, we introduce the \emph{top differential trace} and prove that it precisely defines the singular locus of reduced equidimensional local or graded rings. Motivated by this, we introduce and investigate \emph{nearly regular rings}, which are rings whose top differential trace contains the maximal ideal.
\end{abstract}

\maketitle

\tableofcontents

\section{Introduction} \label{Introduction}

For a commutative ring $S$ and an $S$-module $M$, the ideal
$$\tr_S(M) \defeq  \sum_{f\in \Hom_S(M,S)} f(M)$$
of $S$ is called the trace ideal of $M$.
In recent years, trace ideals have been studied extensively, following Lindo's systematic development of their basic theory~\cite{lindo2017trace}.
In particular, Herzog--Hibi--Stamate~\cite{herzog2019trace} introduced the notion of \emph{nearly Gorenstein rings}, specifically, Cohen--Macaulay local rings whose maximal ideal is contained in the trace ideal of the canonical module.
Their definition is motivated by the fact that, for a Cohen--Macaulay local ring $S$ with canonical module $\omega_S$, one has $\tr_S(\omega_S)=S$ if and only if $S$ is Gorenstein.
More generally, the non-Gorenstein locus of $S$ is precisely the closed subset of $\Spec(S)$ defined by $\tr_S(\omega_S)$.
Since then, the trace ideal of the canonical module has attracted attention as an invariant measuring how far a ring is from being Gorenstein, and has been investigated from various perspectives~(\cites{herzog2019trace,DaoKobayashiTakahashi2021,FicarraHerzogStamateTrivedi2024,Ficarra2025,MiyashitaVarbaro2025,KumashiroMiyashita2025,Kimura2026Schubert}).

Motivated by this, we shift our focus to the modules of differentials.
Let $A$ be a local ring essentially of finite type over a perfect field $\kk$, with residue field $\kappa(A)$, and set $\delta_{\kk}(A) \defeq  \dim A+\trdeg_{\kk}\kappa(A)$. For a Noetherian non-negatively graded algebra $R$ with $R_0=\kk$, we set $\delta_{\kk}(R) \defeq  \dim R$.

In commutative algebra and algebraic geometry, \emph{the module of differentials} $\Omega^1_{A/\kk}$ and its exterior powers $\Omega^i_{A/\kk} \defeq  \bigwedge_A^i \Omega^1_{A/\kk}$ are objects of fundamental importance.
It is classically known that $A$ is regular if and only if $\Omega^1_{A/\kk}$ is a free $A$-module of the appropriate rank~(see \cite[Corollary~16.21]{E}); moreover, if $A$ is equidimensional, the appropriate Fitting ideal of $\Omega^1_{A/\kk}$ defines the singular locus of $A$ (see \cite[Corollary 16.21]{E}).
Thus, the module of differentials naturally encodes the singularities of the ring.
On the other hand, the top exterior power $\Omega^{\delta_{\kk}(A)}_{A/\kk}$ recovers the canonical module when $A$ is a regular local ring.
Therefore, even in the singular case, $\Omega^{\delta_{\kk}(A)}_{A/\kk}$ can be viewed as a natural analog of the canonical module.
This perspective leads us to investigate its trace ideal, which we call \emph{the top differential trace}. The corrected degree is unchanged under localization of an equidimensional algebra by the dimension formula~(\cite[\S13.1, Theorem~A and Corollary~13.4]{E}); it equals $\dim A$ when $\kappa(A)/\kk$ is finite.

In this paper, we study this perspective from both the local and graded points of view. 
Following \cite{costa1979Polynomial}, for a $\ZZ_{\ge 0}$-graded ring and complete local rings, we define their polynomial rank and formal power series rank as follows:

\begin{definition}
\begin{enumerate}[\rm (1)]
\item (\cite{costa1979Polynomial}) Let $R = \bigoplus_{n \geq 0} R_n$ be a $\ZZ_{\geq 0}$-graded ring. 
The \emph{polynomial rank} of $R$, denoted by $\prk(R)$, is the supremum of the set of integers $n \geq 0$ such that $R$ is isomorphic, as a graded ring, to a graded polynomial extension $T[x_1,\dots,x_n]$ of some graded subring $T$ of $R$.
\item Let $S$ be a complete local $S_0$-algebra. 
The \emph{formal power series rank} of $S$, denoted by $\fprk(S)$, is the supremum of the set of integers $n \geq 0$ such that $S$ is isomorphic, as an $S_0$-algebra, to a formal power series ring $T[\![x_1, \dots, x_n]\!]$ over some complete local $S_0$-subalgebra $T$.
\end{enumerate}
\end{definition}

With this terminology in place, our main results on local or graded rings are as follows.
The following statements summarize \autoref{thm:EquivTrRank}, \autoref{thm:GradedCaseTrRank}, \autoref{thm:LocalEquivRegular}, and \autoref{thm:verynice2}.

\begin{theorem}\label{Theorem:OMEDETOU1}
Let $\kk$ be a field of characteristic zero. Let $R$ be a Noetherian non-negatively graded ring with $R_0=\kk$, and let $(A,\mm_A)$ be a local $\kk$-algebra essentially of finite type. Then the following hold.
\begin{enumerate}[\rm(a)]
\item For $0\le i\le\dim R$, one has $\tr_R(\Omega^i_{R/\kk})=R$ if and only if $\prk(R)\ge i$. If $[A/\mm_A:\kk]<\infty$, then for $0\le i\le\dim A$ one has $\tr_{\widehat A}(\widehat\Omega^i_{\widehat A/\kk})=\widehat A$ if and only if $\fprk(\widehat A)\ge i$.
\item Assume that $R$ and $A$ are reduced. For $S=R$ or $A$, one has $\tr_S(\Omega^{\delta_{\kk}(S)}_{S/\kk})=S$ if and only if $S$ is regular. If $S$ is also equidimensional, then
\[
V\!\left(\tr_S(\Omega^{\delta_{\kk}(S)}_{S/\kk})\right)=\Sing(S),\qquad \sqrt{\tr_S(\Omega^{\delta_{\kk}(S)}_{S/\kk})}\subseteq\sqrt{\tr_S(\omega_S)}.
\]
\end{enumerate}
\end{theorem}

We recall that for a module $M$ over a local or graded ring $S$, the condition $\tr_S(M) = S$ is equivalent to $M$ possessing a free summand of rank at least one (\autoref{rem1}).
Thus, part~(b) of \autoref{Theorem:OMEDETOU1} asserts that the regularity of $S$ is equivalent to the existence of a free summand in the top exterior power of differentials $\Omega^{\delta_{\kk}(S)}_{S/\kk}$.
While the classical Jacobian criterion requires the freeness of the module $\Omega^1_{S/\kk}$, our result shows that a single free summand in the top exterior power is sufficient.
We can give a geometric proof of this equivalence under equidimensionality (\autoref{sect4}).

For complete one-dimensional local domains with a coefficient field of characteristic zero, a surjection from the completed module of differentials to the ring already implies regularity; see \cite[Remark~4.2]{maitra2022partial}.
Freeness and torsion of higher differential forms have been studied in \cites{graf2015generalized,miller2019co}; our results relate free summands of exterior powers to polynomial and formal power series splittings.

We note that several other ideals defining the singular locus are already known.
Classically, as mentioned above, the singular locus is defined by the appropriate Fitting ideal of $\Omega^1_{A/\kk}$.
More recently, Iyengar and Takahashi proved that, for a broad class of rings, \emph{the cohomological annihilator} also defines the singular locus~(\cite[Theorem 1.1]{iyengar2016annihilation}).
In contrast to these, \autoref{Theorem:OMEDETOU1} shows that, just as the trace ideal of the canonical module measures the difference from being Gorenstein, the trace ideal of the top exterior power of differentials measures the discrepancy from being regular.
By a technique known as Vasconcelos's trick~(\cite[Remark 3.3]{vasconcelos1991computing}), trace ideals can be easily computed from free presentations; in particular, the top differential trace is computable, for instance, in \texttt{Macaulay2} (\cite{M2}).

Let \(S\) be a commutative ring which is \(R\) or \(A\) in \autoref{Theorem:OMEDETOU1}.
Following the nearly Gorenstein rings introduced in Herzog--Hibi--Stamate (\cite{herzog2019trace}), we say that \(S\) is \emph{nearly regular} if \(\tr_S(\Omega_{S/\kk}^{\delta_{\kk}(S)})\supseteq \mathfrak{m}_S\), without assuming that \(S\) is Cohen--Macaulay.
By \autoref{Theorem:OMEDETOU1}, this notion can be viewed as a regular variant of nearly Gorenstein rings.
The final part of this paper investigates how close nearly regular rings are to regular rings: the condition is mild in dimension one, since every one-dimensional Noetherian graded algebra over a field of characteristic zero is nearly regular, but becomes much more restrictive for tensor products, where it is equivalent to regularity.
In contrast, fiber products provide non-regular examples: the fiber product of two reduced graded rings is nearly regular precisely when both factors are nearly regular and have the same dimension.
Finally, for Stanley--Reisner rings, we characterize nearly regularity by the underlying simplicial complex.

\subsection*{Outline}
In \autoref{sect2}, we review the basic facts on trace ideals that will be used throughout this paper.
In \autoref{sect3}, we study the trace ideals of exterior powers of modules of differentials over local or graded rings. We prove that the equality of this trace ideals with the whole ring characterizes the polynomial rank (\autoref{thm:GradedCaseTrRank} and \autoref{t1hm:EquivRegular1}).
In \autoref{sect5}, building on the results of \autoref{sect3}, we prove that for reduced equidimensional local or graded rings whose residue fields have characteristic $0$, the closed subset defined by the top differential trace coincides with the singular locus (\autoref{thm:verynice2}).
In \autoref{sect6}, based on the results up to \autoref{sect5}, and from the perspective of \cite[Definition 2.2]{herzog2019trace}, we introduce the notion of \emph{nearly regular rings} and investigate their properties.
In \autoref{sect4}, we give a geometric proof of \autoref{thm:LocalEquivRegular} under equidimensionality, using Nash transforms.

\begin{setup}\label{setup1}
Throughout this paper, we use the following notation:
\begin{itemize}
\item Let $\ZZ_{\geq 0}$ be the monoid of non-negative integers;
\item Let $R = \bigoplus_{i \geq 0} R_i$ be a $\ZZ_{\geq 0}$-graded ring. Unless otherwise specified, we do not assume $R$ is Noetherian for the sake of generality. If the reader is only interested in the Noetherian case, it may be assumed throughout;
\item Whenever we assume that $R_0$ is a local ring with unique maximal ideal $\mfrakm_0$, $R$ is equipped with the unique graded maximal ideal $\mfrakm_R \defeq  \mfrakm_0 + \bigoplus_{i > 0} R_i$. When there is no risk of confusion about $R$, we simply write $\mathfrak{m}_R$ as $\mathfrak{m}$;
\item Let \(A\) be a ring with the unique maximal ideal \(\mfrakm\). We say that \(A\) is {\it complete} if the canonical morphism \(A \to \widehat{A} \defeq  \lim_{n \geq 0} A/\mfrakm^n\) is an isomorphism;
\item For a morphism of rings $A \to B$, let $\Omega^1_{B/A}$ denote the module of differentials of $B$ over $A$, which is a $B$-module. For each integer $k \geq 0$, we define the $k$-th module of differentials as the wedge product $\Omega^k_{B/A} \defeq  \bigwedge_B^k \Omega^1_{B/A}$.
\end{itemize}
\end{setup}

\subsection{Acknowledgments}
The authors would like to thank Takanori Nagamine for introducing us to the literature on locally nilpotent derivations. They are also grateful to Kazufumi Eto, Kaito Kimura, Shinya Kumashiro, Naoyuki Matsuoka, Yuki Mifune, Yuya Otake, and Keiichi Watanabe for helpful discussions and valuable comments.
The first-named author was supported by JSPS KAKENHI Grant number 24KJ1085. The second-named author was supported by JSPS KAKENHI Grant number 25KJ1744.

\section{Preliminaries}\label{sect2}
The purpose of this section is to lay the groundwork for the discussions of our main results.

\begin{notation}
    Throughout this section, we follow \autoref{setup1} and assume that \(R_0\) has a unique maximal ideal \(\mfrakm_0\) unless otherwise specified. 
    Note that, since we do not assume the condition on \(R_0\) other than this, setting \(R_i = 0\) for \(i > 0\) is just a theory of local rings. So all the results in this section are valid for local rings and their modules.
\end{notation}

\subsection{Trace ideals}

\begin{definition}
For a graded $R$-module $M$, the sum of all images of homomorphisms $\phi \in \Hom_R(M,R)$ is called the {\it trace} of $M$:
\[
\tr_R(M)\defeq  \sum_{\phi \in \Hom_R(M,R)}\phi(M).
\]
\end{definition}

\begin{remark}\label{rem:interestingfiniteness}
Let $M$ be a graded $R$-module. Then, we have 
\[
\tr_R(M)=\sum_{\phi \in {}^*\Hom_R(M,R)}\phi(M).
\]
by \cite[Remark 2.2]{KumashiroMiyashita2025}, where \({}^*\Hom_R(M, R)\) is the internal hom of graded \(R\)-modules.
Note that throughout their paper, they assume the Noetherian property of \(R\) but this result does not require this.
\end{remark}

\begin{lemma}\label{rem1}
Let $M$ and $N$ be graded $R$-modules.
The following hold:
\begin{enumerate}[\rm (1)]
\item If there exists a graded surjective $R$-homomorphism  $M\twoheadrightarrow N$, then $\tr_R(M)\supseteq \tr_R(N)$;
\item $\tr_R(\wedge^k M) \supseteq \tr_R(\wedge^{k+1} M)$
holds for any $k>0$;
\item If $M$ is finitely presented, then $\tr_{R_\mfrakp}(M_\mfrakp)=\tr_R(M)R_\mfrakp$~(see \cite[Proposition 2.8~(viii)]{lindo2017trace});
\item $\tr_R(M)=R$
if and only if
$M$ has an $R$-free summand;\footnote{The Noetherian case appeared in \cite[Remark 2.2~(4)]{celikbas2023traces}.}
\item $\tr_R(M \oplus N)=R$
if and only if $\tr_R(M)=R$ or $\tr_R(N)=R$.
\end{enumerate}
\end{lemma}

\begin{proof}
(1) is clear by definition.

(2):
Note that $\wedge^{k+1} M \cong (\wedge^{k} M) \wedge M \cong \left((\wedge^{k} M) \otimes_R M\right)/L$ for some $R$-module $L\subseteq (\wedge^{k} M) \otimes_R M$.
Thus,
we have
$\tr_R(\wedge^{k+1} M) \subseteq \tr_R\left(\wedge^{k} M \otimes_R M\right)$ by (1).
Then
we obtain
$\tr_R\left(\wedge^{k} M \otimes_R M\right)
\subseteq
\tr_R(\wedge^k M) \cap \tr(M)
\subseteq
\tr_R(\wedge^k M)$
by \cite[Proposition 1.3]{herzog2019trace}.

(4): If \(M\) has an \(R\)-free summand, then \(\tr_R(M) = R\) by taking a retraction of an \(R\)-free summand of \(M\).
Conversely, we can follow the proof of \cite[Proposition 2.8(3)]{lindo2017trace} as follows: Assume there exist finitely many homogeneous \(R\)-module homomorphisms \(\alpha_i \colon M \to R\) and homogeneous elements \(x_i\) of \(M\) indexed by a finite set \(I\) with an equality \(\sum_{i \in I} \alpha_i(x_i) = 1\) in \(R\).
Since \(R\) has only one graded maximal ideal \(\mfrakm_{R_0} + R_+\), there exists \(j \in I\) such that \(\alpha_j(x_j) \notin \mfrakm_{R_0} + R_+\).
Since $\alpha_j(x_j)$ is homogeneous and lies outside $\mfrakm_{R_0}+R_+$, it belongs to $R_0^\times$. Hence \(\alpha_j \colon M \to R\) is surjective. Choose $x\in M$ with $\alpha_j(x)=1$. The map $R\to M$, $r\mapsto rx$, splits $\alpha_j$, so $M$ has an $R$-free summand.

(5):
If $\tr_R(M)=R$ or $\tr_R(N)=R$, then
$\tr_R(M\oplus N)=\tr_R(M)+\tr_R(N)=R$.
Assume $\tr_R(M)\neq R$ and $\tr_R(N)\neq R$.
Then $\tr_R(M)$ and $\tr_R(N)$ are proper graded ideals of $R$.
Since $R$ is local$^{*}$, every proper graded ideal is contained in $\mm_R$.
Hence
$\tr_R(M)+\tr_R(N)\subseteq \mm_R$,
and in particular $\tr_R(M\oplus N)\neq R$.
\end{proof}

\begin{lemma} \label{MinimalKill}
Let \(R\) be a Noetherian ring and let \(M\) be a finite \(R\)-module.
If \(M_{\mfrakp}=0\) for all \(\pp \in \Ass(R)\), then \(\tr_R(M) = 0\) holds.
\end{lemma}
\begin{proof}
    Let \(\mfrakp\) be an associated prime ideal of \(R\). Then there exists an element \(r_{\mfrakp} \in R \setminus \mfrakp\) such that \(r_{\mfrakp} \cdot M = 0\).
    Therefore \(\Ann_{R}(M) \nsubseteq \mfrakp\) holds and the prime avoidance shows that
    \begin{equation*}
        \Ann_R(M) \nsubseteq \bigcup_{\mfrakp \in \Ass(R)} \mfrakp.
    \end{equation*}
    The right hand side is the set of zero divisors of \(R\) and then a non zero divisor \(r\) of \(R\) kills \(M\).
    This shows \(\tr_R(M) = 0\).
\end{proof}

Finally, we prove that, for suitable graded rings, the trace ideals of the exterior powers of the module of differentials vanish for all degrees greater than \(\dim(R)\).

\begin{proposition}\label{prop:3.8}
Assume that $R$ is a reduced Noetherian graded ring and that $R_0=\kk$ is a perfect field.
Then we have \(\tr_R(\Omega_{R/\kk}^i) = 0\) for all \(i > \dim(R)\).
\end{proposition}
\begin{proof}
    Take any
$\mfrakp \in \Ass(R)$.
    Note that $\mfrakp$ is a minimal prime ideal and $R_{\mfrakp} \cong \Frac(R/\mfrakp)$ because $R$ is reduced.
    Then we have
    \begin{equation*}
        (\Omega_{R/\kk}^1)_{\mfrakp} \cong \Omega_{R_{\mfrakp}/\kk}^1 \cong \Omega_{\Frac(R/\mfrakp)/\kk}^1 \cong \bigoplus_{i=1}^{\dim (R/\mfrakp)} \Frac(R/\mfrakp),
    \end{equation*}
where the last isomorphism follows from the equality
$\trdeg_{\kk} \Frac(R/\mfrakp) = \dim R/\mfrakp$  since $R/\mfrakp$ is a finitely generated affine domain over $\kk$,
    together with the description of the module of differentials of finitely generated field extensions over a perfect field.
    This shows that \((\Omega^i_{R/\kk})_{\mfrakp}=0\) for any \(i > \dim (R/\mfrakp)\).
    Because of \(\dim(R) \geq \dim(R/\mfrakp)\), we can conclude that \((\Omega^i_{R/\kk})_{\mfrakp} = 0\) for any \(i > \dim(R)\).
    By \autoref{rem1}~(3), the above argument shows \((\tr_R(\Omega^i_{R/\kk}))_{\mfrakp} = 0\) for any \(i > \dim(R)\).
    By \autoref{MinimalKill}, \(\tr_R(\Omega^i_{R/\kk})=0\) for any \(i > \dim(R)\).
\end{proof}

\begin{remark}
Even if $R$ is a domain, $\Omega_{R/\kk}^i$ does not necessarily vanish when $i > \dim(R)$.
Assume that $\kk$ is an algebraically closed field of characteristic $0$.
For each integer $a \ge 1$, consider
\[
R_a
\defeq 
\kk[x_1,y_1,z_1,\dots,x_a,y_a,z_a]/
(x_1y_1-z_1^2,\dots,x_ay_a-z_a^2)
\cong \bigotimes_{i=1}^a\kk[x_i,y_i,z_i]/(x_iy_i-z_i^2),
\]
where $\deg(x_i)=\deg(y_i)=\deg(z_i)=1$ for each $1\le i\le a$.
Then $\Omega^{p}_{R_a/\kk} \neq 0$
for all $\dim R_a+1 \le p \le 3a$.
Indeed, since $R_a$ is a graded normal domain with the graded maximal ideal $\mm_a$,
by \cite[Remark 3.2]{miller2019co},
one has
$\tor \left(\Omega^{p}_{(R_a)_{\mfrakm_a}/\kk} \right)\neq 0$
for all $\dim (R_a)_{\mm_a}+1 \le p \le 3a$,
where $\operatorname{tor}(\Omega^{p}_{(R_a)_{\mfrakm_a}/\kk})$ is the torsion submodule of $\Omega^{p}_{(R_a)_{\mfrakm_a}/\kk}$.
In particular,
we have
$\Omega^{p}_{R_a/\kk} \neq 0$
for all $\dim R_a+1 \le p \le 3a$.
\end{remark}

\section{Trace ideals of exterior powers of differentials}\label{sect3}

In this section, we study the relationship between the trace ideals of exterior powers of the module of differentials and some invariants which measure how close the rings are to regular (see \autoref{thm:EquivTrRank}).
In particular, under reduced assumption, we prove that the equality of the top differential trace with the whole ring characterizes the regularity of the ring (see \autoref{t1hm:EquivRegular1}).

\begin{setup} \label{SetupSection3}
    Let \(S\) be an algebra over a local ring \(S_0\). Furthermore, we assume one of the following conditions:
    \begin{enumerate}
        \item \(S\) is a \(\setZ_{\geq 0}\)-graded ring whose degree \(0\)-part is \(S_0\). In this case, we will write \(S = \bigoplus_{i \geq 0} S_i\). Every homogeneous unit of \(S\) belongs to \(S_0^{\times}\);
        \item \(S\) is a Noetherian ring with the unique maximal ideal \(\mfrakm_S\) and a local subring \(S_0\) whose inclusion is local, and $\Omega^1_{(S/\mfrakm_S)/S_0}=0$.
    \end{enumerate}
    We often say the former case is that \(S\) is graded, and the latter case is that \(S\) is local, for simplicity.
    If \(S\) is local, we will write its completion as \(\widehat{S} \defeq  \lim_{n \geq 1} S/\mfrakm_S^nS\) and we say \(S\) is complete if the canonical morphism \(S \to \widehat{S}\) is an isomorphism.
\end{setup}

\begin{definition}\label{DefCompletedModDiff}
Assume that \(S\) is local. The \emph{completed module of differentials} is the $\widehat S$-module
\begin{equation*}
    \widehat\Omega^1_{S/S_0}=\varprojlim_n\Omega^1_{S/S_0}/\mm_S^n\Omega^1_{S/S_0} \cong \varprojlim_n (\Omega^1_{S/S_0} \otimes_{S} S/\mfrakm_S^n).
\end{equation*}
We write $\widehat\Omega^i_{S/S_0}=\bigwedge^i_{\widehat S}\widehat\Omega^1_{S/S_0}$. When $\Omega^1_{S/S_0}$ is finite, this is also the \(\mfrakm_S\)-adic completion of $\Omega^i_{S/S_0}$.
\end{definition}

We begin by recalling the definition of polynomial rank.

\begin{definition}[\cite{costa1979Polynomial}] \label{DefPolyRank}
Assume that \(S\) is graded.
For \(n\ge0\), a \emph{graded polynomial extension of \(n\) variables} of \(S\) is a polynomial ring
$S[x_1,\dots,x_n]$
equipped with a grading compatible with that of \(S\) such that each \(x_i\) is homogeneous with \(\deg(x_i) > 0\).
The \emph{polynomial rank} of \(S\) is the supremum of the set of integers \(n\ge0\) such that \(S\) is isomorphic, as a graded ring, to a graded polynomial extension \(T[x_1,\dots,x_n]\) of some graded subring \(T\).
Following \cite{costa1979Polynomial}, we denote the polynomial rank of \(S\) by $\prk(S)$.
\end{definition}

Since we do not assume Noetherian condition in general, we need the following notion.

\begin{definition}[\cite{costa1979Polynomial}] \label{DefcoPolyRank}
Assume that \(S\) is graded.
The \emph{copolynomial rank} of \(S\) is the infimum of the Krull dimensions of graded rings \(T\) for which there exists an isomorphism
$S \cong T[x_1,\dots,x_n]$
of graded rings for some \(n\ge0\), where \(T[x_1, \dots, x_n]\) is a graded polynomial extension of \(T\).
We denote the copolynomial rank of \(S\) by \(\cprk(S)\).
\end{definition}

\begin{remark}\label{rem:prank}
By the definition, we know the following for the graded ring \(S\).
\begin{enumerate}[\rm (1)]
\item We have \(\prk(S)\le \dim(S)\).
    \item If \(S\) is Noetherian and \(\dim(S)<\infty\), then
$\cprk(S)=\dim(S)-\prk(S)$.
\end{enumerate}
\end{remark}

\begin{definition}
    Let \(S\) be a complete local \(S_0\)-algebra.
    The \emph{formal power series rank} of \(S\) is the supremum of the set of integers \(n \geq 0\) such that \(S\) is isomorphic, as an \(S_0\)-algebra, to a formal power series ring \(T[|x_1, \dots, x_n|]\) over some complete local \(S_0\)-subalgebra \(T\). We denote the formal power series rank of \(S\) by \(\fprk(S)\).

    The \emph{co-formal power series rank} of \(S\) is the infimum of the Krull dimensions of complete local \(S_0\)-subalgebras \(T\) for which there exists an isomorphism \(S \cong T[|x_1, \dots, x_n|]\) of \(S_0\)-algebras for some \(n \geq 0\). We denote the co-formal power series rank of \(S\) by \(\cfprk(S)\).
\end{definition}

\begin{remark}\label{rem:fprank}
By the definition, we know the following for complete local \(S_0\)-algebra \(S\):
\begin{enumerate}[\rm (1)]
\item We have \(\fprk(S)\le \dim(S)\).
    \item If \(S\) is Noetherian and \(\dim(S)<\infty\), then
$\cfprk(S)=\dim(S)-\fprk(S)$.
\end{enumerate}
\end{remark}

\begin{lemma} \label{CompatibleCompletion}
    Assume that \(S\) is local and complete.
    For an \(S\)-module \(M\), set \(\widehat{M} \defeq  \allowbreak \lim_{n \geq 1} M/\mfrakm_S^nM\).
    Then the trace ideals \(\tr_S(M)\) and \(\tr_S(\widehat{M})\) are the same in \(S\).
    In particular, \(\tr_S(\Omega^1_{S/S_0}) = S\) if and only if \(\tr_{S}(\widehat{\Omega^1}_{S/S_0}) = S\).
\end{lemma}

\begin{proof}
    Since any \(S\)-linear morphism \(M \to S\) can be extended to \(\widehat{M} \to \widehat{S} \cong S\), the containment \(\tr_S(\widehat{M}) \supseteq \tr_S(M)\) holds. For the converse, set $I=\tr_S(M)$ and let $\psi:\widehat M\to S$ be $S$-linear. For $z\in\widehat M$ and $n\ge1$, choose $m_n\in M$ with $z-\iota(m_n)\in\mfrakm_S^n\widehat M$, where $\iota:M\to\widehat M$ is the canonical map; this is possible by \citeSta{05GG}. Then $\psi(z)\in I+\mfrakm_S^n$ for every $n$, because the image of $\psi\iota:M\to S$ belongs to $I$. Krull's intersection theorem applied to $S/I$ gives $\bigcap_{n\ge1}(I+\mfrakm_S^n)=I$, proving the reverse inclusion.
\end{proof}

\begin{lemma} \label{CompletionOmega}
    Assume that \(S\) is local.
    Then we have the canonical isomorphism
    \begin{equation*}
        \Omega^1_{S/S_0} \otimes_S S/\mfrakm_S^n \xrightarrow{\cong} \Omega^1_{\widehat{S}/S_0} \otimes_S S/\mfrakm_S^n
    \end{equation*}
    of \(S\)-modules.
    In particular, we have an isomorphism
    \begin{equation*}
        \widehat{\Omega}^1_{S/S_0} \xrightarrow{\cong} \widehat{\Omega}^1_{\widehat{S}/S_0}
    \end{equation*}
    of completed modules of differentials.
\end{lemma}

\begin{proof}
    By \citeSta{02HQ}, we have canonical isomorphisms
    \begin{align*}
        \Omega^1_{S/S_0} \otimes_S S/\mfrakm_S^n & \xrightarrow{\cong} \Omega^1_{(S/\mfrakm_S^{n+1})/S_0} \otimes_{S/\mfrakm_S^{n+1}} S/\mfrakm_S^n \quad \text{and} \\
        \Omega^1_{\widehat{S}/S_0} \otimes_{\widehat{S}} \widehat{S}/\mfrakm_S^n\widehat{S} & \xrightarrow{\cong} \Omega^1_{(\widehat{S}/\mfrakm_S^{n+1}\widehat{S})/S_0} \otimes_{\widehat{S}/\mfrakm_S^{n+1}\widehat{S}} \widehat{S}/\mfrakm_S^n\widehat{S}
    \end{align*}
    of \(S\)-modules.
    Since \(\widehat{S}/\mfrakm_S^n\widehat{S}\) is isomorphic to \(S/\mfrakm_S^n\) (see, for example, \citeSta{05GG}), the right-hand sides of the isomorphisms are canonically isomorphic each other. This completes the proof.
\end{proof}

\begin{proposition} \label{prop:Yeahhhh}
    Keep the setting of \autoref{SetupSection3}. Then the following hold:
    \begin{enumerate}[\rm (1)]
        \item If $\tr_S(\Omega^{i}_{S/S_0}) = S$ for some \(i \geq 0\), then $\tr_S(\Omega^{j}_{S/S_0}) = S$ for any $0 \le j \le i$;
        \item The equality \(\tr_S(\Omega^1_{S/S_0}) = S\) holds if and only if there exists a derivation \(D \colon S \to S\) over \(S_0\) with a slice \(t \in S\), namely, \(D(t) = 1\). Moreover, we can take \(D\) and \(t\) satisfying the following conditions;
        \begin{enumerate}[\rm(a)]
            \item if \(S\) is graded, then \(t \in S_+\) and \(D\) is a locally nilpotent homogeneous derivation \(D \colon S \to S(-a)\) with \(a \in \setZ_{> 0}\);
            \item if \(S\) is local and complete, then \(t\) belongs to the maximal ideal of \(S\).
        \end{enumerate}
        \item If we assume \(S \supseteq \setQ\), then the equality \(\tr_S(\Omega^1_{S/S_0}) = S\) holds if and only if \(\prk(S) \geq 1\) when \(S\) is graded, or \(\fprk(S) \geq 1\) when \(S\) is local and complete.
    \end{enumerate}
\end{proposition}

\begin{proof}
    (1) follows from \autoref{rem1}~(2); the case $j=0$ is immediate.

(2):
Assume that there exists a derivation \(D \colon S \to S\) over \(S_0\) and a slice \(t \in S\).
By the universality of the differential module \(\Omega^1_{S/S_0}\), the derivation \(D\) factors through the universal derivation \(d \colon S \to \Omega^1_{S/S_0}\) and an \(S\)-linear homomorphism \(\varphi \colon \Omega^1_{S/S_0} \to S\).
Since \(t \in S\) is a slice, we know \(\varphi(dt) = 1\) and this shows \(\tr_S(\Omega^1_{S/S_0}) = S\).

Assume \(\tr_S(\Omega^1_{S/S_0}) = S\), namely, that $\Omega^1_{S/S_0}$ has an $S$-free summand. Choose a projection $\pi:\Omega^1_{S/S_0}\twoheadrightarrow S$ and set $D_0=\pi\circ d$.

If $S$ is local, set $\kappa=S/\mfrakm_S$. The conormal sequence $\mfrakm_S/\mfrakm_S^2\to\Omega^1_{S/S_0}\otimes_S\kappa\to\Omega^1_{\kappa/S_0}\to0$ has zero last term by \autoref{SetupSection3}. Since $\pi\otimes\kappa$ is surjective, so is the composition \(\mfrakm_S/\mfrakm_S^2 \xrightarrow{d} \Omega^1_{S/S_0} \otimes_S \kappa \xrightarrow{\pi \otimes \kappa} \kappa\) and thus there is $t\in\mfrakm_S$ such that $u \defeq  D_0(t)$ is a unit. Then $D \defeq  u^{-1}D_0$ is an $S_0$-derivation and $D(t)=1$.

If $S$ is graded, take homogeneous components in an expression for $1 \in \tr_S(\Omega^1_{S/S_0})$, as in \autoref{rem:interestingfiniteness}. Since $S_0$ is local, we may choose $\pi$ homogeneous and a homogeneous $t\in S$ with $u \defeq \pi(dt)\in S_0^\times$. Because $d(S_0)=0$, one has $a \defeq \deg t > 0$ and $D_0$ has degree $-a$. Set $D \defeq u^{-1}D_0$. Then $D(t)=1$, and $D$ is locally nilpotent: for homogeneous $f\in S$, $D^n(f)=0$ whenever $na>\deg f$.

(3): Assume \(S \supseteq \setQ\).
If \(S\) is graded, the graded slice theorem (see \cite[Proposition 2.1]{wright1981jacobian}) shows $S\cong (\ker D)[t]$ when \(\tr_S(\Omega^1_{S/S_0}) = S\).
Conversely, if \(\prk(S) \geq 1\), then \(S \cong T[x]\) for some invariable \(x\) and a graded sub \(S_0\)-algebra \(T\) of \(S\). Then \(\Omega^1_{S/S_0} \cong (\Omega^1_{T/S_0} \otimes_T S) \bigoplus S dx\) shows \(\tr_S(\Omega^1_{S/S_0}) = S\).

Assume \(S\) is local and complete. If \(\tr_S(\Omega^1_{S/S_0}) = S\), applying \cite{zariski1965Studies}*{Lemma 4}, the existence of the derivation \(D\) and a slice \(t\) imply that \(S\) is isomorphic to a formal power series ring \(\Ker(D)[|t|]\) with \(S_0 \subseteq \Ker(D)\). This shows \(\fprk(S) \geq 1\).
Conversely, if $S\cong T[|x|]$ as an $S_0$-algebra, differentiation with respect to $x$ is an $S_0$-derivation $D:S\to S$ satisfying $D(x)=1$. By part~(2), $\tr_S(\Omega^1_{S/S_0})=S$.
\end{proof}

\begin{theorem}\label{thm:EquivTrRank}
Let $S$ be an algebra over a ring $S_0$ satisfying \autoref{SetupSection3}. 
Assume that $S$ contains $\setQ$ and has finite Krull dimension $d \defeq  \dim(S)$. 
Let $k$ be an integer with $0 \le k \le d$.
\begin{enumerate}[\rm (1)]
    \item If $S$ is graded and $\prk(S) \ge d-k$, then $\tr_S(\Omega^{d-k}_{S/S_0}) = S$. Conversely, if $\tr_S(\Omega^{d-k}_{S/S_0}) = S$, then $\cprk(S) \le k$.
    \item Assume that $S$ is local and essentially of finite type over $S_0$, and that $S_0$ is Noetherian.
    Let $\widehat{S}$ be the $\mfrakm_S$-adic completion of $S$. 
    Then $\fprk(\widehat{S}) \ge d-k$ if and only if $\tr_{\widehat{S}}(\widehat{\Omega}^{d-k}_{\widehat{S}/S_0}) = \widehat{S}$. 
\end{enumerate}
\end{theorem}

\begin{proof}
Set $q \defeq d-k$. We may assume $d-k > 0$, since $\tr_S(\Omega^0_{S/S_0}) = \tr_S(S) = S$ holds trivially.

(1): We first prove the case that \(S\) is graded.
Assume $\prk(S) \ge d-k$. Then there exists a graded $S_0$-algebra $T$ and a graded isomorphism $S \cong T[x_1, \dots, x_{d-k}]$ where each $x_i$ has positive degree. 
The module of differentials decomposes as $\Omega^1_{S/S_0} \cong (\Omega^1_{T/S_0} \otimes_T S) \oplus \bigoplus_{i=1}^{d-k} S dx_i$. 
Taking the $(d-k)$-th exterior power, the property of exterior powers of direct sums yields:
\begin{equation*}
    \Omega^{d-k}_{S/S_0} \cong \bigoplus_{i+j=d-k} \left( \bigwedge^i_S (\Omega^1_{T/S_0} \otimes_T S) \otimes_S \bigwedge^j_S \Big(\bigoplus_{m=1}^{d-k} S dx_m\Big) \right).
\end{equation*}
For $i=0$ and $j=d-k$, we obtain the $S$-free direct summand $\bigwedge^{d-k}_S (\bigoplus_{m=1}^{d-k} S dx_m) \cong S(dx_1 \wedge \dots \wedge dx_{d-k})$. 
Hence, $\tr_S(\Omega^{d-k}_{S/S_0}) = S$.

Conversely, assume $\tr_S(\Omega^{d-k}_{S/S_0}) = S$. By \autoref{rem1}~(2), since $d-k > 0$, we have $\tr_S(\Omega^1_{S/S_0}) = S$. 
Since $S \supseteq \setQ$, \autoref{prop:Yeahhhh}~(3) implies $\prk(S) \ge 1$. 
Let $e = \prk(S)$. We can write $S \cong T[x_1, \dots, x_e]$ for some graded subalgebra $T$ such that $\prk(T) = 0$. 
Setting $F \defeq  \bigoplus_{m=1}^e S dx_m$, we have:
\begin{align*}
\Omega^{d-k}_{S/S_0} = \bigwedge_S^{d-k}\Bigl((\Omega^1_{T/S_0}\otimes_T S) \oplus F\Bigr) 
&\cong \bigoplus_{j=0}^{{\min \{e, d-k\}}} \left(\Big(\bigwedge\nolimits_T^{d-k-j}\Omega^1_{T/S_0}\Big)\otimes_T S\right)\otimes_S \Big(\bigwedge\nolimits_S^{j}F\Big) \\
&\cong \bigoplus_{j=0}^{{\min \{e, d-k\}}}\; \bigoplus_{|J|=j}\; \left( \Big(\bigwedge\nolimits_T^{d-k-j}\Omega^1_{T/S_0}\Big)\otimes_T S\Big(-\sum_{k\in J}\deg(x_k)\Big) \right).
\end{align*}
Since $\tr_S(\Omega^{d-k}_{S/S_0}) = S$, \autoref{rem1}~(4) and (5) imply that there exists an integer $i$ such that $(\Omega^{d-k-i}_{T/S_0})\otimes_T S$ has an $S$-free direct summand. 
Taking the base change along the surjection $S \twoheadrightarrow T$ reveals that $\bigwedge^{d-k-i}_T \Omega^1_{T/S_0}$ has a $T$-free direct summand, so $\tr_T(\Omega^{d-k-i}_{T/S_0}) = T$. 
If $\dim(T) > k$, then $d-k-i \ge d-k-e \geq \dim(T)-k > 0$. 
By \autoref{prop:Yeahhhh}~(1), this would force $\tr_T(\Omega^1_{T/S_0}) = T$, meaning $\prk(T) \ge 1$, which contradicts $\prk(T) = 0$. 
Therefore, $\dim(T) \le k$, yielding $\cprk(S) \le k$.

(2): Finite generation and \autoref{CompletionOmega} give $\widehat{\Omega}^1_{\widehat{S}/S_0} \cong \Omega^1_{S/S_0} \otimes_S \widehat{S}$ and $\widehat{\Omega}^q_{\widehat{S}/S_0}=\bigwedge_{\widehat{S}}^q \widehat{\Omega}^1_{\widehat{S}/S_0}$.
In a representation $\widehat{S}=T[\![x_1,\ldots,x_e]\!]$, the ring $T \cong \widehat{S}/(x_1,\ldots,x_e)$ is complete Noetherian local with the same residue field. In particular, it satisfies \autoref{SetupSection3}. Its completed module of differentials \(\widehat{\Omega}^1_{T/S_0}\) is finite over \(T\): the conormal sequence \(\mfrakm_T/\mfrakm_T^2 \to \widehat{\Omega}^1_{T/S_0} \otimes_T S/\mfrakm_S \to \widehat{\Omega}^1_{(S/\mfrakm_S)/S_0} \to 0\) bounds the dimension of the middle term by $\dim_{T/\mm_T}\mm_T/\mm_T^2$, and the completeness lifts the generators. The universal property for continuous derivations therefore gives
\begin{equation*}
\widehat{\Omega}^1_{\widehat{S}/S_0} \cong(\widehat\Omega^1_{T/S_0}\otimes_T \widehat{S})\oplus\bigoplus_{j=1}^e \widehat{S}\,dx_j.
\end{equation*}
If $e\ge q = d-k$, the wedge of $q$ variable differentials is a free summand, proving one implication.

Conversely, set $e\defeq\fprk(\widehat S)$, choose a representation $\widehat S=T[\![x_1,\ldots,x_e]\!]$, and assume $\tr_{\widehat S}(\widehat\Omega^q_{\widehat S/S_0})=\widehat S$.
The same direct-sum and retraction argument gives $\tr_T(\widehat\Omega^{q-j}_{T/S_0})=T$ for some $j\le e$. If $e<q$, \autoref{rem1}(2) and \autoref{CompatibleCompletion} imply $\tr_T(\Omega^1_{T/S_0})=T$. Then \autoref{prop:Yeahhhh} splits another formal variable from $T$, contradicting maximality. Thus $e\ge q$.
\end{proof}

\begin{corollary}\label{thm:GradedCaseTrRank}
Assume that \(S\) is graded.
Let $k$ be an integer with $0 \le k \leq \dim(S)$.
Assume that $S$ is Noetherian and contains $\setQ$.
Then \(\prk(S) \ge \dim(S)-k\) if and only if \(\tr_S(\Omega^{\dim(S)-k}_{S/S_0}) = S\).
In particular, $S$ is isomorphic to a graded polynomial ring over a graded \(0\)-dimensional $S_0$-algebra if and only if $\tr_S(\Omega^{\dim(S)}_{S/S_0}) = S$.
\end{corollary}
\begin{proof}
It follows from \autoref{thm:EquivTrRank} and \autoref{rem:prank}~(2).
\end{proof}

\begin{corollary}\label{t1hm:EquivRegular1}
Assume that \(S\) is graded and moreover \(S\) is a Noetherian reduced \(\mathbb{Q}\)-algebra of finite Krull dimension such that \(S_0\) is a regular local ring.
Then the following hold:
\begin{enumerate}[\rm (1)]
\item Assume that $S_0$ is a field. Then \(S\) is regular if and only if \(\tr_S(\Omega^{\dim(S)}_{S/S_0}) = S\);
\item Assume that $S$ is a domain. Then \(S\) is regular if and only if \(\tr_S(\Omega^{\dim(S)-\dim(S_0)}_{S/S_0}) = S\).
\end{enumerate}
\end{corollary}

\begin{proof}
Notice that, under this assumption, $S$ is regular if and only if $S$ is isomorphic to a polynomial ring over $S_0$ in $\dim(S)-\dim(S_0)$ variables.
Then (1) follows from \autoref{thm:GradedCaseTrRank}.

(2):
By \autoref{thm:GradedCaseTrRank}, the equality
$\tr_S(\Omega^{\dim(S)-\dim(S_0)}_{S/S_0})=S$
is equivalent to the inequality
$\prk(S)\ge \dim(S)-\dim(S_0)$.
Thus it suffices to show that $S$ is regular if and only if
$\prk(S)\ge \dim(S)-\dim(S_0)$.

If $S$ is regular, then $S$ is isomorphic, as a graded ring, to a polynomial ring over $S_0$ in $\dim(S)-\dim(S_0)$ homogeneous variables. Hence $\prk(S)\ge \dim(S)-\dim(S_0)$.

Conversely, assume that $\prk(S)\ge \dim(S)-\dim(S_0)$. Then there exist a graded subring $T\subseteq S$ and a graded ring isomorphism
$S\cong T[x_1,\dots,x_N]$
for some $N\ge \dim(S)-\dim(S_0)$, where each $x_i$ is homogeneous of positive degree.

Since $S$ is Noetherian, $T$ is also Noetherian.
Hence $\dim(S)=\dim(T)+N$, so $\dim(T)=\dim(S)-N\le \dim(S_0)$.

We claim that $T=S_0$. Otherwise, $T_+$ is a nonzero prime ideal of the domain $T$, since $T/T_+\cong S_0$ is a domain. A chain of primes of length $\dim S_0$ in $S_0$ lifts to a chain in $T$ starting at $T_+$. Prepending $(0)\subsetneq T_+$ gives $\dim T\ge\dim S_0+1$, a contradiction. Thus $T=S_0$.
It follows that $S\cong S_0[x_1,\dots,x_N]$. Since $\dim(S)=\dim(S_0)+N$, we have $N=\dim(S)-\dim(S_0)$. Therefore $S$ is a polynomial ring over
$S_0$, and hence $S$ is regular.
\end{proof}

\begin{remark}
The domain assumption in \autoref{t1hm:EquivRegular1}~(2) is essential.
It cannot be replaced by the weaker assumption that \(R\) is reduced.
For example, let \(R=\mathbb{Q}[[t]][u]/(tu)\), where \(\deg(t)=0\) and \(\deg(u)=1\), so that \(R_0=\mathbb{Q}[[t]]\).
Then \(R\) is reduced, \(\dim(R)=\dim(R_0)=1\), and hence
\(\Omega^{\dim(R)-\dim(R_0)}_{R/R_0}=\Omega^0_{R/R_0}=R\).
Thus \(\tr_R(\Omega^{\dim(R)-\dim(R_0)}_{R/R_0})=R\), whereas \(R\) is not regular.
\end{remark}

\begin{corollary}\label{thm:LocalEquivRegular}
Let $\kk$ be a field of characteristic zero, and let $A$ be a reduced local $\kk$-algebra essentially of finite type. Set $\delta_{\kk}(A)=\dim A+\trdeg_{\kk}\kappa(A)$. Then $A$ is regular if and only if $\tr_A(\Omega^{\delta_{\kk}(A)}_{A/\kk})=A$.\footnote{Appendix~\ref{sect4} gives an alternative geometric proof under equidimensionality, using Nash transforms and Nobile's theorem.}
\end{corollary}

\begin{proof}
If $A$ is regular, then it is essentially smooth over $\kk$ and $\Omega^1_{A/\kk}$ is free of rank $\delta_{\kk}(A)$ by the Jacobian criterion \cite[Corollary~16.21]{E}. Hence its top exterior power has unit trace.

Conversely, first suppose that $\kappa(A)/\kk$ is algebraic, hence finite separable, so $\delta_{\kk}(A)=\dim A$ and \autoref{SetupSection3} applies. Trace commutes with completion for finite modules (\cite{lindo2017trace}*{Proposition~2.8 (viii)}), and \autoref{thm:EquivTrRank} gives $\widehat A\cong T[\![x_1,\ldots,x_{\delta_{\kk}(A)}]\!]$ with $\dim T=0$. Excellence and reducedness of $A$ make $\widehat A$, and hence $T$, reduced. Thus $T$ is a field, and regularity descends from $\widehat A$ to $A$.

In general, write $A=B_{\pp}$ with $B$ a reduced affine $\kk$-algebra. For each minimal prime $\mfrakq\subseteq\pp$ of $B$,
we have
\[\dim(B/\mfrakq)=\dim(A/\mfrakq A)+\trdeg_{\kk}\kappa(A)\le \delta_{\kk}(A).\]
The condition \(\tr_A(\Omega^{\delta_{\kk}(A)}_{A/\kk}) = A\) implies $\Omega^{\delta_{\kk}(A)}_{\Frac(B/\mfrakq)/\kk}\ne0$, so characteristic zero forces equality \(\dim(B/\mfrakq) = \delta_{\kk}(A)\).
After a principal localization away from $\pp$, remove the other components and arrange $\tr_B(\Omega^{\delta_{\kk}(A)}_{B/\kk})=B$, using \autoref{rem1}~(3). The resulting $B$ is equidimensional of dimension $\delta_{\kk}(A)$. Every localization of \(B\) by maximal ideals has dimension $\delta_{\kk}(A)$ by \cite[\S13.1, Theorem~A]{E}, and its residue field is finite over $\kk$, so the case already proved makes the localizations regular. Regularity localizes; hence $A$ is regular.
\end{proof}

The term \(\trdeg_{\kk} \kappa(A)\) in \(\delta_{\kk}(A)\) is necessary. For example, $A=(\kk[x,y,z]/(xy))_{(x,y)}$ is singular and has dimension one, but $\Omega^1_{A/\kk}$ has the free summand $A\,dz$ and hence $\tr_A\Omega_{A/\kk}^1=A$. Here $\kappa(A)=\kk(z)$ and $\delta_{\kk}(A)=2$.

\section{Top differential trace and the singular locus}\label{sect5}

In this section, we investigate the relationship between the algebraic variety defined by the trace ideal of top differentials and the singular locus. We prove that they coincide when the residue field has characteristic \(0\) and the ring is reduced and equidimensional (see \autoref{thm:verynice2}).

\begin{setup}\label{setup4}
Let $\kk$ be a field. We consider either a Noetherian non-negatively graded ring $S$ with $S_0=\kk$, or a local $\kk$-algebra $S$ essentially of finite type. Write $\mm_S$ for the graded maximal ideal in the former case and the maximal ideal in the latter. Set $\delta_{\kk}(S)=\dim S$ in the graded case and $\delta_{\kk}(S)=\dim S+\trdeg_{\kk}(S/\mm_S)$ in the local case. If $S$ is equidimensional,
$\trdeg_{\kk}\Frac(S/\mfrakq)=\delta_{\kk}(S)$ for every minimal prime $\mfrakq$, and
$\delta_{\kk}(S_{\pp})=\delta_{\kk}(S)$
for every
$\pp\in\operatorname{Spec}S$.
\end{setup}

\begin{remark}\label{rem:canonical-module-exists}
Under \autoref{setup4}, \(S\) admits a canonical module $\omega_S$.
Indeed, in the local case, \(S\) is a homomorphic image of a regular local ring, hence of a Gorenstein local ring. In the graded case, \(S\) is a homomorphic image of a polynomial ring over \(\kk\), hence of a Gorenstein ring. Therefore \(S\) admits a canonical module. Moreover, in the graded case, \(S\) admits a graded canonical module
~(see \cite[Definition~(2.1.2)]{goto1978graded}).
\end{remark}

\begin{lemma}\label{ann0}
Assume that $S$ is a domain and set $K\defeq  \Frac(S)$.
Then
$\Omega_{K/\kk}^i\neq 0$,
$\Ann_S\!\Bigl(\Omega_{S/\kk}^{i}\Bigr)=(0)$
and $\tr_S \Bigl(\Omega_{S/\kk}^{i}\Bigr) \neq (0)$
for every integer $i$ with $0\le i\le\delta_{\kk}(S)$.
\end{lemma}
\begin{proof}
Set $M\defeq  \Omega_{S/\kk}^i$.
Since the module of differentials and exterior powers commute with localization, we have
$M\otimes_S K
\cong
\Omega_{K/\kk}^i$.
Since we know $\trdeg_{\kk}K=\delta_{\kk}(S)$ in both cases of \autoref{setup4} and $K/\kk$ is a finitely generated field extension, one has $\dim_K\Omega^1_{K/\kk}\ge \trdeg_\kk(K) \ge i$.
Hence $\dim_K\Omega^1_{K/\kk}\ge i$, and therefore $\Omega_{K/\kk}^i\neq 0$.
Consequently $M\otimes_S K\neq 0$.

Now let $a\in \Ann_S(M)$. Tensoring the equality $aM=0$ with $K$ gives
$(a/1)\cdot(M\otimes_S K)=0$.
Since $M\otimes_S K$ is a nonzero $K$-vector space, this implies $a/1=0$ in $K$.
As $S$ is a domain,
we have $a=0$.
Thus $\Ann_S(M)=(0)$.
By \cite[Proposition 1.2.3 (b)]{BH}, this implies that
$\tr_S(M) \neq (0)$.
\end{proof}

\begin{proposition}\label{lem:ann1}
Assume that $S$ is reduced and equidimensional. Then $\tr_S(\Omega^{\delta_{\kk}(S)}_{S/\kk})_{\pp}\ne0$ for every $\pp\in\operatorname{Spec}S$.
\end{proposition}
\begin{proof}
Choose a minimal prime $\mfrakq\subseteq\pp$. The field $K=S_{\mfrakq}=\Frac(S/\mfrakq)$ has transcendence degree $\delta_{\kk}(S)$ over $\kk$, so $\Omega^{\delta_{\kk}(S)}_{K/\kk}\ne0$ by \autoref{ann0}. Its trace is $K$. Hence $\tr_S(\Omega^{\delta_{\kk}(S)}_{S/\kk})_{\mfrakq}\ne0$ by \autoref{rem1}~(3), and therefore $\tr_S(\Omega^{\delta_{\kk}(S)}_{S/\kk})_{\pp}\ne0$.
\end{proof}

\begin{theorem}\label{thm:verynice2}
Assume that $S$ is reduced and equidimensional and that $\operatorname{char}\kk=0$.
Then
\[V\!\left(\tr_S(\Omega^{\delta_{\kk}(S)}_{S/\kk})\right)=\operatorname{Sing}(S).\]
\end{theorem}
\begin{proof}
For every prime $\pp$,
trace and exterior powers of differentials commute with localization, so
$\tr_S(\Omega^{\delta_{\kk}(S)}_{S/\kk})S_{\pp}=\tr_{S_{\pp}}(\Omega^{\delta_{\kk}(S)}_{S_{\pp}/\kk})$.
By \autoref{setup4}, $\delta_{\kk}(S)=\delta_{\kk}(S_{\pp})$. Hence \autoref{thm:LocalEquivRegular} says that this ideal is $S_{\pp}$ if and only if $S_{\pp}$ is regular, proving the assertion.
\end{proof}

Recall that a Noetherian local ring \(A\) admitting a canonical module \(\omega_A\) is called \emph{quasi-Gorenstein} if \(\omega_A \cong A\).

\begin{corollary}\label{cor:verynice2}
Assume \(S\) is reduced, equidimensional, and contains \(\mathbb{Q}\).
Then
$\sqrt{\tr_S(\Omega^{\delta_{\kk}(S)}_{S/\kk})} \subseteq \sqrt{\tr_S(\omega_S)}$.
\end{corollary}
\begin{proof}
Since \(S\) is equidimensional, we have
\[
V(\tr_S(\omega_S))
=
\{\pp \in \Spec(S) : S_{\pp} \text{ is not quasi-Gorenstein}\}
\]
by \cite[(1.7)]{aoyama1983some} and \cite[Corollary 3.4]{aoyama1985endomorphism}. We note that the same conclusion also holds in the graded case by \cite[Remark 2.9~(2), (5)]{KumashiroMiyashita2025}.
Thus, by \autoref{thm:verynice2}, we obtain
$
V(\tr_S(\omega_S))
\subseteq \Sing(S)
=
V\!\Bigl(\tr_S\!\bigl(\Omega^{\delta_{\kk}(S)}_{S/\kk}\bigr)\Bigr).
$
Thus the assertion follows.
\end{proof}

\begin{remark}\label{rem:equidimensionality-is-necessary}
The equidimensionality assumption in \autoref{thm:verynice2} cannot be omitted. For example, let
$S\defeq \QQ[x,y,z]/(xy,xz)$.
Then \(S\) is reduced
but it is not equidimensional
because \((xy,xz)=(x)\cap(y,z)\).
Moreover, we can check \(\dim S=2\).
A direct computation, or the formula proved later in \autoref{thm:fiber-top-trace-exact-formula}, gives
$\tr_S(\Omega_{S/\QQ}^{2})=(y,z)S$.
Hence
$V\!\left(\tr_S(\Omega_{S/\QQ}^{2})\right)=V(y,z)$.
However, the prime \((y,z)S\) is not singular, since \(S_{(y,z)}\cong \QQ(x)\) is a field. Thus \(V(\tr_S(\Omega_{S/\QQ}^{2})) \supsetneq \Sing(S)\). In fact, \(\Sing(S)=V(x,y,z)\).
\end{remark}

We conclude this section by proving a useful proposition on equidimensionality.

\begin{proposition}\label{prop:radical-trace-equidimensional}
Assume $\operatorname{char}\kk=0$. If $\mm_S\subseteq\sqrt{\tr_S(\Omega^{\delta_{\kk}(S)}_{S/\kk})}$, then $S$ is equidimensional.
\end{proposition}
\begin{proof}
Let $S_{\red} = S/\sqrt{0_S}$. Every $\kk$-derivation $D:S\to S$ preserves $\sqrt{0_S}$: if $a^n=0$, then $0=D^n(a^n)=n!(D(a))^n+ab$ for some $b\in S$, so $D(a)$ is nilpotent. For $i>0$, a map $f:\Omega^i_{S/\kk}\to S$ defines an alternating multiderivation by sending \((a_1, \dots, a_i)\) to $f(da_1\wedge\cdots\wedge da_i)$. Fixing all but one component shows that it descends modulo $\sqrt{0_S}$ to a map $\Omega^i_{S_{\red}/\kk}\to S_{\red}$.
Consequently, $\tr_S(\Omega^i_{S/\kk})S_{\red}\subseteq\tr_{S_{\red}}(\Omega^i_{S_{\red}/\kk})$.
This also holds for $i=0$. The hypothesis therefore passes to $S_{\red}$. Since reduction preserves dimensions of irreducible components and the residue field, it also preserves $\delta_{\kk}$; hence we may assume that $S$ is reduced.

The case $\dim S=0$ is immediate. Otherwise, a minimal prime $\pp$ with $\dim(S/\pp)<\dim S$ satisfies $\trdeg_{\kk}\Frac(S/\pp)=\delta_{\kk}(S/\pp)<\delta_{\kk}(S)$.
Thus $(\Omega^{\delta_{\kk}(S)}_{S/\kk})_\pp=0$. Using \autoref{rem1}~(3) gives $\tr_{S}(\Omega^{\delta_{\kk}(S)}_{S/\kk})_\pp=0$, hence $\tr_S(\Omega^{\delta_{\kk}(S)}_{S/\kk}) \subseteq \pp$. This contradicts $\mm_S\subseteq\sqrt{\tr_S(\Omega^{\delta_{\kk}(S)}_{S/\kk})}\subseteq\pp\subsetneq\mm_S$.
\end{proof}

\section{Nearly regular rings}\label{sect6}

In this section, motivated by \autoref{Theorem:OMEDETOU1} established in the preceding sections and by the perspective of \cite[Definition 2.2]{herzog2019trace}, we introduce the notion of nearly regular rings as a generalization of regular rings, and investigate their ring-theoretic properties.
Throughout this section, let \(S\) be a commutative ring satisfying \autoref{setup4}.

\begin{definition}
The ring $S$ is \emph{nearly regular} (over $\kk$) if $\mm_S\subseteq\tr_S(\Omega^{\delta_{\kk}(S)}_{S/\kk})$. This definition does not require $S$ to be Cohen--Macaulay. It uses degree $\delta_{\kk}(S) = \dim S$ for graded algebras and \(\delta_{\kk}(S) = \dim S + \trdeg_{\kk}\kappa(S)\) for local algebras. 
\end{definition}

\begin{remark}\label{rem:OKOKOKOK}
Assume that $S$ is reduced and $\operatorname{char}\kk=0$. Then:
\begin{enumerate}[\rm (1)]
\item $\mm_S\subseteq\sqrt{\tr_S(\Omega^{\delta_{\kk}(S)}_{S/\kk})}$ if and only if $S$ is equidimensional and regular on $\Spec S\setminus\{\mm_S\}$;
\item if $S$ is nearly regular, then $S$ is equidimensional and regular on $\Spec S\setminus\{\mm_S\}$;
\item $S$ is nearly regular if and only if $\mm_S\subseteq\tr_S(\Omega^i_{S/\kk})$ for every $0\le i\le\delta_{\kk}(S)$.
\end{enumerate}
\end{remark}
\begin{proof}
Part~(1) follows from \autoref{prop:radical-trace-equidimensional} and \autoref{thm:verynice2}; part~(2) follows from (1). Part~(3) follows from \autoref{rem1}~(2), with $i=0$ immediate.
\end{proof}

\begin{proposition}\label{kusa2}
Let \(\kk\) be a field of characteristic zero, and let \(S\) be a Noetherian non-negatively graded ring with \(S_0=\kk\).
Then \(\mfrakm_S\subseteq\tr_S(\Omega^1_{S/\kk})\) holds. Moreover this becomes equality if and only if \(\prk(S)=0\).
\end{proposition}
\begin{proof}
The Euler derivation \(E:S\to S\), defined by \(E(f)=(\deg f)f\) for homogeneous \(f\), induces an \(S\)-linear map \(\Omega^1_{S/\kk}\to S\).
Since every positive integer is invertible in \(\kk\), its image is \(\mfrakm_S\).
Thus \(\mfrakm_S\subseteq\tr_S(\Omega^1_{S/\kk})\).
By \autoref{prop:Yeahhhh}~(3), the trace is proper if and only if \(\prk(S)=0\).
Since \(\mfrakm_S\) is maximal, the assertion follows.
\end{proof}

\begin{corollary}
Every one-dimensional \(\ZZ_{\ge 0}\)-graded Noetherian algebra over a field of characteristic zero with degree-zero part equal to the field is nearly regular.
\end{corollary}

\begin{remark}\label{OK1}
Write $S=\kk[X_1,\dots,X_m]/(f_1,\dots,f_r)$, or a localization of such an algebra, and let $J=\bigl(\frac{\partial f_j}{\partial X_i}\bigr)_{i,j}$ be its $m\times r$ Jacobian matrix.
Then we have $\tr_S(\Omega^1_{S/\kk})=I_1(X)$, where $X$ is a matrix whose columns generate $\ker(J^t)$.
\begin{proof}
Note that there is a free presentation
$S^{\oplus r} \xrightarrow{J} S^{\oplus m} \to \Omega^1_{S/\kk} \to 0$.
Hence the assertion follows from Vasconcelos' trick; see \cite[Remark 3.3]{vasconcelos1991computing} (or \cite[Proposition 3.1]{herzog2019trace}).
\end{proof}
\end{remark}

\begin{example}
For a weighted graded ring $W$ over $W_0=\ZZ$,
the inclusion $\tr_W(\Omega^1_{W/\ZZ}) \supseteq W_+$ does not hold in general. Indeed, let $W=\ZZ[t^2,t^3]$, where $\deg(t^2)=2$ and $\deg(t^3)=3$. Using the method of \autoref{OK1} together with a computation in \texttt{Macaulay2}~\cite{M2}, one obtains
$\tr_W(\Omega^1_{W/\ZZ})=(2t^2,t^4,t^3)$.
On the other hand, for $W_p=\mathbb F_p[t^2,t^3]$ with $p>3$, one has $\tr_{W_p}(\Omega^1_{W_p/\mathbb F_p})=\mm_{W_p}$. For $p=3$ the trace is $W_3$, because $\Omega^1_{W_3/\mathbb F_3}\cong W_3\,d(t^2)\oplus (W_3/(t^3))\,d(t^3)$.
\end{example}

In the remainder of this section, we investigate how nearly regular relates to other ring-theoretic operations.

\begin{remark}
The following hold:
\begin{enumerate}
\item It is known that nearly Gorenstein rings remain nearly Gorenstein after quotienting by a regular sequence
(\cite[Proposition 2.3]{herzog2019trace}).
In contrast, even a quotient of a regular ring by a non-zerodivisor need not be nearly regular.
For instance, \(T=\QQ[x,y,z]\) is regular, but
\(R=T/(x^3+y^3+z^3)\) is not nearly regular.
Indeed, using \autoref{OK1}, a computation in \texttt{Macaulay2} shows that
\(\tr_R(\Omega_{R/R_0}^2)=(x^2,y^2,z^2)R \neq \mm_R\);
\item It is known that every Veronese subring of a standard graded (nearly) Gorenstein ring of positive Krull dimension is nearly Gorenstein
(see \cite[Corollary 4.7]{herzog2019trace} and \cite[Corollary 3.7]{Miyashita2024LinearGeneralization}).
In contrast, a Veronese subring of a standard graded polynomial ring need not be nearly regular.
For example, let \(T=\QQ[x,y]\) be the standard graded polynomial ring, and let
\(R=\QQ[x^3,x^2y,xy^2,y^3]\) be its third Veronese subring.
Then, using \autoref{OK1}, a computation in \texttt{Macaulay2} shows that
\(\tr_R(\Omega_{R/R_0}^2)=\mm_R^2 \neq \mm_R\).
\end{enumerate}
\end{remark}

\subsection{The top differential trace ideal for tensor products}

In this subsection, we give an explicit formula
\autoref{thm:tensor-product-exact-trace-formula}
for the trace of the top exterior power of the module of differentials of tensor products over a field in the reduced case. At the same time, we prove that, in characteristic zero, the nearly regularity of the tensor product is equivalent to regularity.

\begin{remark}\label{rem:tensor-product-differentials}
Let \(T\) be a ring, and let \(A\) and \(B\) be \(T\)-algebras. Then there is a canonical isomorphism
\[
\Omega^1_{(A\otimes_T B)/T}\cong (\Omega^1_{A/T}\otimes_T B)\oplus (A\otimes_T \Omega^1_{B/T}).
\]
Indeed, this is the affine case of \cite[Exercise II.8.3(a)]{hartshorne2013algebraic}.
\end{remark}

\begin{theorem}\label{thm:tensor-product-exact-trace-formula}\label{cor:tensor-product-nearly-regular}
Let $\kk$ be a perfect field, and let $A$ and $B$ be reduced $\ZZ_{\ge 0}$-graded Noetherian rings with $A_0=B_0=\kk$. Assume that $\dim A,\dim B>0$ and that $R\defeq A\otimes_\kk B$ is reduced. Then the following hold:
\begin{enumerate}[\rm(1)]
\item We have
$\tr_R(\Omega^{\dim R}_{R/\kk})=\tr_A(\Omega^{\dim A}_{A/\kk})R\cdot\tr_B(\Omega^{\dim B}_{B/\kk})R=\tr_A(\Omega^{\dim A}_{A/\kk})R\cap\tr_B(\Omega^{\dim B}_{B/\kk})R$.
\item If $\operatorname{char}\kk=0$, then the following conditions are equivalent:
\begin{enumerate}[\rm(a)]
\item $R$ is nearly regular;
\item $R$ is regular;
\item $A$ and $B$ are regular;
\item $\mm_R\subseteq\sqrt{\tr_R(\Omega^{\dim R}_{R/\kk})}$.
\end{enumerate}
\end{enumerate}
\end{theorem}

\begin{proof}
Set $d_A\defeq \dim A$, $d_B\defeq \dim B$, $d\defeq \dim R$, $I_A\defeq \tr_A(\Omega^{d_A}_{A/\kk})$, and $I_B\defeq \tr_B(\Omega^{d_B}_{B/\kk})$.

(1): We have $d=d_A+d_B$ by \cite[Theorem~(2.2.5)]{goto1978graded}. By \autoref{rem:tensor-product-differentials} and the exterior-power decomposition of a direct sum,
$\Omega^d_{R/\kk}\cong\bigoplus_{i=0}^d\left(\Omega^{d-i}_{A/\kk}\otimes_\kk\Omega^i_{B/\kk}\right)$.
Since the modules of differentials are finitely presented, \cite[Proposition~4.1(ii)]{herzog2019trace} gives
\[
\tr_R(\Omega^d_{R/\kk})=\sum_{i=0}^d\tr_A(\Omega^{d-i}_{A/\kk})R\cdot\tr_B(\Omega^i_{B/\kk})R.
\]
If $i<d_B$, then $d-i>d_A$, whereas if $i>d_B$, then $i>\dim B$. Thus \autoref{prop:3.8} shows that every summand except $i=d_B$ vanishes, proving $\tr_R(\Omega^d_{R/\kk})=I_A R\cdot I_B R$. The equality with $I_A R\cap I_B R$ follows from \cite[Proposition~4.1(ii)]{herzog2019trace}.

(2): Assume that $\operatorname{char}\kk=0$. By \autoref{t1hm:EquivRegular1} and (1), condition (c) gives $I_A=A$ and $I_B=B$, hence $\tr_R(\Omega^d_{R/\kk})=R$ and therefore (b). The implications (b)$\Rightarrow$(a)$\Rightarrow$(d) follow from \autoref{t1hm:EquivRegular1} and the definition of nearly regularity.

It remains to prove (d)$\Rightarrow$(c). Suppose that $I_A\ne A$. Since $I_A$ is homogeneous, $I_A\subseteq\mm_A$, and (1) gives $\tr_R(\Omega^d_{R/\kk})\subseteq\mm_A R$. The ideal $\mm_A R$ is radical because $R/\mm_A R\cong B$ is reduced. Hence (d) implies $\mm_R\subseteq\mm_A R$. Passing to the quotient forces $\mm_B=0$, contradicting $\dim B>0$. Thus $I_A=A$, and similarly $I_B=B$. Both factors are regular by \autoref{t1hm:EquivRegular1}.
\end{proof}

\begin{corollary}[{cf. \cite[Corollary 4.4]{herzog2019trace}}]\label{cor:polynomial-extension-nearly-regular}
Let \(\kk\) be a field of characteristic \(0\), let \(R\) be a reduced \(\ZZ_{\ge 0}\)-graded Noetherian ring with \(R_0=\kk\), and set \(S\defeq R[x_1,\dots,x_n]\) for some \(n\ge 1\), where each \(x_i\) is homogeneous of positive degree. Then $\mm_S\subseteq\sqrt{\tr_S(\Omega^{\dim S}_{S/\kk})}$ if and only if $S$ is nearly regular, if and only if $R$ is regular.
\end{corollary}

\begin{proof}
If \(\dim R=0\), then \(R=\kk\), since \(R\) is reduced and \(R_0=\kk\). Hence \(S\) is a polynomial ring over \(\kk\), and the assertion is clear.

Since \(S\cong R\otimes_\kk \kk[x_1,\dots,x_n]\),
the assertion follows from \autoref{cor:tensor-product-nearly-regular}.
\end{proof}

\subsection{The top differential trace ideal for fiber products, nearly regular Stanley--Reisner rings}

The purpose of this subsection is to prove an explicit formula for the trace of the top differential module of a fiber product arising from reduced graded rings, and to characterize its nearly regularity. As a consequence, we characterize the nearly regularity of Stanley--Reisner rings in terms of the corresponding simplicial complexes.

Throughout this subsection, let \(\kk\) be a perfect field, and let \(A\) and \(B\) be \(\mathbb Z_{\ge 0}\)-graded Noetherian rings with \(A_0=B_0=\kk\). We write \(\mfrakm_A\defeq A_{>0}\) and \(\mfrakm_B\defeq B_{>0}\). Let
\[
\varepsilon_A:A\twoheadrightarrow A/\mfrakm_A\cong \kk
\quad\text{and}\quad
\varepsilon_B:B\twoheadrightarrow B/\mfrakm_B\cong \kk
\]
be the canonical graded surjections, and set
$R\defeq A\times_{\kk} B\defeq \{(a,b)\in A\times B \mid \varepsilon_A(a)=\varepsilon_B(b)\}$.
We denote by \(\pi_A:R\twoheadrightarrow A\) and \(\pi_B:R\twoheadrightarrow B\) the canonical projections, and by \(\iota_A:A\to R\), \(a\mapsto (a,\varepsilon_A(a))\), and \(\iota_B:B\to R\), \(b\mapsto (\varepsilon_B(b),b)\), the canonical sections. We also set \(K_A\defeq \ker(\pi_A)=\mfrakm_B R\) and \(K_B\defeq \ker(\pi_B)=\mfrakm_A R\). Note that \(\mfrakm_R=R_{>0}=K_A\oplus K_B=\mfrakm_A R\oplus \mfrakm_B R\), that \(K_AK_B=0\), and that \(\dim R=\max\{\dim A,\dim B\}\).

\begin{proposition}\label{prop:fiber-differentials}
There is an isomorphism
$\Omega^1_{R/\kk}\cong \bigl((\Omega^1_{A/\kk}\otimes_A R)\oplus (\Omega^1_{B/\kk}\otimes_B R)\bigr)/J$,
where \(J\) is the \(R\)-submodule generated by the elements \(b\,da+a\,db\) with \(a\in \mfrakm_A\) and \(b\in \mfrakm_B\).
\end{proposition}

\begin{proof}
Set \(P\defeq A\otimes_{\kk} B\) and \(I\defeq \mfrakm_A\otimes_{\kk}\mfrakm_B\). Since \(A=\kk\oplus \mfrakm_A\) and \(B=\kk\oplus \mfrakm_B\) as \(\kk\)-vector spaces, we have \(P\cong \kk\oplus \mfrakm_A\oplus \mfrakm_B\oplus (\mfrakm_A\otimes_{\kk}\mfrakm_B)\), and hence \(R\cong P/I\). Applying the conormal exact sequence to \(P\twoheadrightarrow R=P/I\), we obtain
\[
I/I^2 \longrightarrow \Omega^1_{P/\kk}\otimes_P R \longrightarrow \Omega^1_{R/\kk}\longrightarrow 0.
\]
By \autoref{rem:tensor-product-differentials}, \(\Omega^1_{P/\kk}\cong (\Omega^1_{A/\kk}\otimes_{\kk} B)\oplus (A\otimes_{\kk}\Omega^1_{B/\kk})\), so \(\Omega^1_{P/\kk}\otimes_P R\cong (\Omega^1_{A/\kk}\otimes_A R)\oplus (\Omega^1_{B/\kk}\otimes_B R)\). Thus \(\Omega^1_{R/\kk}\cong \bigl((\Omega^1_{A/\kk}\otimes_A R)\oplus (\Omega^1_{B/\kk}\otimes_B R)\bigr)/\operatorname{Im}(\delta)\), where \(\delta:I/I^2\to \Omega^1_{P/\kk}\otimes_P R\) is the first map of the conormal sequence. For \(a\in \mfrakm_A\) and \(b\in \mfrakm_B\), the class of \(a\otimes b\in I\) is sent to \(d(a\otimes b)\otimes 1\), and \(d(a\otimes b)=da\otimes b+a\otimes db\), which corresponds to \(b\,da+a\,db\). Since \(I=\mfrakm_A\otimes_{\kk}\mfrakm_B\) is generated by the elements \(a\otimes b\), the assertion follows.
\end{proof}

\begin{definition}[{cf. \cite[Definition 3.9]{KumashiroMiyashita2025}}]\label{def:dagger-ideal}
Let \(S\) be a \(\mathbb Z_{\ge 0}\)-graded ring with \(S_0=\kk\), and let \(\mfrakm_S\defeq S_{>0}\). For an ideal \(I\subseteq S\), we define
\[
I^\dagger\defeq 
\begin{cases}
I & \text{if } I\neq S,\\
\mfrakm_S & \text{if } I=S.
\end{cases}
\]
\end{definition}

\begin{lemma}\label{lem:maps-into-maximal-ideal}
Let \(S\) be a \(\mathbb Z_{\ge 0}\)-graded Noetherian ring with \(S_0=\kk\), let \(\mfrakm_S\defeq S_{>0}\), and let \(M\) be a finitely generated graded \(S\)-module. Assume \(\mfrakm_S\subseteq \tr_S(M)\). Then every \(x\in\mfrakm_S\) can be written as
$x=\sum_{\ell=1}^n \varphi_\ell(m_\ell)$
for some \(S\)-linear maps \(\varphi_\ell:M\to S\) and elements \(m_\ell\in M\), with
$\varphi_\ell(M)\subseteq \mfrakm_S$
for all \(\ell\).
\end{lemma}
\begin{proof}
Since \(S/\mfrakm_S\cong \kk\), any ideal containing \(\mfrakm_S\) is either \(\mfrakm_S\) or \(S\). Thus \(\tr_S(M)\) is either \(\mfrakm_S\) or \(S\).

If \(\tr_S(M)=\mfrakm_S\), then every image of an \(S\)-linear map \(M\to S\) is contained in \(\mfrakm_S\). Hence the assertion follows from the definition of the trace ideal.
If \(\tr_S(M)=S\), then \(1\in\tr_S(M)\), so there exist \(S\)-linear maps \(\psi_\ell:M\to S\) and elements \(m_\ell\in M\) such that
$1=\sum_{\ell=1}^n \psi_\ell(m_\ell)$.
For \(x\in\mfrakm_S\), set \(\varphi_\ell\defeq x\psi_\ell\). Then we have
$x=\sum_{\ell=1}^n \varphi_\ell(m_\ell)$,
and \(\varphi_\ell(M)\subseteq xS\subseteq\mfrakm_S\) for every \(\ell\).
\end{proof}

\begin{theorem}[{cf.\ \cite[Theorem~3.11]{KumashiroMiyashita2025}}]\label{thm:fiber-top-trace-exact-formula}\label{cor:fiber-nearly-regular}
Assume that $A$ and $B$ are reduced and that $\dim A,\dim B>0$. Then the following hold:
\begin{enumerate}[\rm(1)]
\item We have
$\tr_R(\Omega^{\dim R}_{R/\kk})=\tr_A(\Omega^{\dim R}_{A/\kk})^\dagger R\oplus\tr_B(\Omega^{\dim R}_{B/\kk})^\dagger R$.
\item $R$ is nearly regular if and only if both $A$ and $B$ are nearly regular and $\dim A=\dim B$.
\end{enumerate}
\end{theorem}

\begin{proof}
Set $d\defeq \dim R=\max\{\dim A,\dim B\}$, $I_A\defeq \tr_A(\Omega^d_{A/\kk})$, and $I_B\defeq \tr_B(\Omega^d_{B/\kk})$.

(1): Via the canonical sections $A\to R$ and $B\to R$, we regard the differentials of elements of $A$ and $B$ as elements of $\Omega^1_{R/\kk}$. By \autoref{prop:fiber-differentials}, $\Omega^d_{R/\kk}$ is generated by wedges of these differentials. We call such a wedge pure if all its entries come from the same factor, and mixed if entries from both factors occur. Let $f:\Omega^d_{R/\kk}\to R$ be $R$-linear.

$R$ is reduced, and its minimal primes come from the two factors. At a minimal prime on the $A$ branch, $\mm_B$ and its differentials vanish; the symmetric assertion holds on the $B$ branch. Thus every mixed wedge vanishes at all minimal primes. Since $R$ injects into the product of these localizations, $f$ annihilates mixed wedges.

Since $A=\kk\oplus\mm_A$, the module $\Omega^1_{A/\kk}$ is generated by $da$ with $a\in\mm_A$. Hence pure $A$ wedges vanish at every minimal prime on the $B$ branch, so their images under $f$ have zero $B$-component. Restricting $f$ along the section $A\to R$ therefore gives an $A$-linear map $\Omega^d_{A/\kk}\to\mm_A$, whose image lies in $I_A^\dagger$. The same argument applies to pure $B$ wedges, yielding
$\tr_R(\Omega^d_{R/\kk})\subseteq I_A^\dagger R\oplus I_B^\dagger R$.

For the reverse inclusion, $I_A^\dagger$ is the sum of the images of all $A$-linear maps $\alpha:\Omega^d_{A/\kk}\to\mm_A$. Indeed, if $I_A\ne A$, then $I_A\subseteq\mm_A$ because $I_A$ is homogeneous; if $I_A=A$, the assertion follows from \autoref{lem:maps-into-maximal-ideal}. Composing these maps with the natural surjection $\Omega^d_{R/\kk}\twoheadrightarrow\Omega^d_{A/\kk}$ and the $R$-linear inclusion $\mm_A\to R$, $a\mapsto(a,0)$, gives $I_A^\dagger R\subseteq\tr_R(\Omega^d_{R/\kk})$. The same argument for $B$ proves the reverse inclusion.

(2): Suppose that $R$ is nearly regular. If $\dim A<\dim B=d$, then $I_A=0$ by \autoref{prop:3.8}. Part (1) therefore gives $\tr_R(\Omega^d_{R/\kk})\subseteq\mm_B R\subsetneq\mm_R$, a contradiction. The case $\dim B<\dim A$ is symmetric. Thus $\dim A=\dim B$.

When $\dim A=\dim B=d$, part (1) and $\mm_R=\mm_A R\oplus\mm_B R$ show that $R$ is nearly regular if and only if $I_A^\dagger\supseteq\mm_A$ and $I_B^\dagger\supseteq\mm_B$. By the definition of $\dagger$, these conditions are equivalent to $I_A\supseteq\mm_A$ and $I_B\supseteq\mm_B$, respectively, which precisely mean that both factors are nearly regular.
\end{proof}

We now recall some notation on simplicial complexes and Stanley--Reisner rings.
Set $V=[n]=\{1,2,\cdots,n\}$.
A nonempty subset $\Delta$ of the power set $2^{V}$ of $V$
is called a \textit{simplicial complex} on $V$ if
$\{v\} \in \Delta$ for all $v \in V$, and $F \in \Delta$,
$H \subset F$ implies $H \in \Delta$.
For a face $F$ of $\Delta$, we set
$$\lk_\Delta(F) \defeq  \{ G \in \Delta \;;\; G\cup F \in \Delta, F\cap G = \emptyset  \}.$$
This complex is called the \textit{link} of $F$.
Let \(\kk\) be a field and let \(T=\kk[x_1,\ldots,x_n]\). The \textit{Stanley--Reisner ideal} of \(\Delta\) is the squarefree monomial ideal
\[
I_\Delta\defeq \left(x_{i_1}x_{i_2}\cdots x_{i_p}\mid 1\leq i_1<\cdots<i_p\leq n,\ \{i_1,\ldots,i_p\}\notin\Delta\right)\subseteq T.
\]
The quotient ring \(\kk[\Delta]\defeq T/I_\Delta\) is called the \textit{Stanley--Reisner ring} of \(\Delta\). Since \(I_\Delta\) is squarefree, \(\kk[\Delta]\) is reduced.

The following result is presumably well known to specialists in Stanley--Reisner rings, but we include a proof for the reader's convenience.

\begin{remark}\label{prop:sr-regular-punctured-simplex}
Let $\kk$ be a field and let \(\Delta\) be a connected simplicial complex on the vertex set \(V\). Then \(\kk[\Delta]\) is regular on the punctured spectrum if and only if \(\Delta\) is a simplex.
\end{remark}

\begin{proof}
Set \(R=\kk[\Delta]\) and \(\mfrakm=(x_v\mid v\in V)\). If \(\Delta\) is a simplex, then \(R\) is a polynomial ring over \(\kk\), and hence \(R\) is regular.

Conversely, assume that \(R\) is regular on \(\Spec R\setminus\{\mfrakm\}\). We first show that \(\lk_\Delta(v)\) is a simplex for every vertex \(v\in V\). Set \(\mfrakp_v=(x_u\mid u\in V\setminus\{v\})R\). Since \(\{v\}\in\Delta\), the ideal \(\mfrakp_v\) is a prime ideal of \(R\), and \(\mfrakp_v\neq\mfrakm\). Thus \(R_{\mfrakp_v}\) is regular. By the standard localization formula for Stanley--Reisner rings, we have
$R_{\mfrakp_v}\cong \kk(x_v)[\lk_\Delta(v)]_{\mfrakn}$,
where \(\mfrakn\) denotes the irrelevant maximal ideal of \(\kk(x_v)[\lk_\Delta(v)]\). Hence \(\kk(x_v)[\lk_\Delta(v)]_{\mfrakn}\) is regular. If \(\Gamma=\lk_\Delta(v)\) and \(r\) is the number of vertices of \(\Gamma\), then the embedding dimension of this local ring is \(r\), while its Krull dimension is \(\dim\Gamma+1\). Therefore \(r=\dim\Gamma+1\), so \(\Gamma\) has a face containing all its vertices. Thus \(\lk_\Delta(v)\) is a simplex.

Now let \(F\) and \(G\) be facets of \(\Delta\) with \(F\cap G\neq\emptyset\). Choose \(v\in F\cap G\). Then \(F\setminus\{v\}\) and \(G\setminus\{v\}\) are facets of \(\lk_\Delta(v)\). Since \(\lk_\Delta(v)\) is a simplex, it has a unique facet, and hence \(F\setminus\{v\}=G\setminus\{v\}\). Thus \(F=G\). Hence any two distinct facets of \(\Delta\) are disjoint. Since \(\Delta\) is connected, this forces \(\Delta\) to have only one facet. Therefore \(\Delta\) is a simplex.
\end{proof}


Finally, as an application of \autoref{cor:fiber-nearly-regular}, we characterize the nearly regularity of Stanley--Reisner rings.

\begin{corollary}\label{cor:sr-nearly-regular}
Let $\kk$ have characteristic zero, and let $\Delta$ be a finite simplicial complex with at least one vertex and connected components $\Delta_1,\ldots,\Delta_n$. Set $S=\kk[\Delta]$ and $d=\dim S$. The following are equivalent:
\begin{enumerate}[\rm (1)]
\item $S$ is nearly regular;
\item $\mm_S\subseteq\sqrt{\tr_S(\Omega^d_{S/\kk})}$;
\item each $\Delta_i$ is a simplex and $\dim\Delta_i=\dim\Delta$ for every $i$.
\end{enumerate}
Under these conditions, $\tr_S(\Omega^d_{S/\kk})$ is $S$ if $n=1$ and $\mm_S$ if $n>1$. Moreover, $\mm_S\subseteq\tr_S(\omega_S)$.
\end{corollary}
\begin{proof}
Clearly (1) implies (2). Assume (2). By \autoref{prop:radical-trace-equidimensional} and \autoref{thm:verynice2}, $S$ is equidimensional and regular away from $\mm_S$. Write $S_i=\kk[\Delta_i]$; then $S\cong S_1\times_\kk\cdots\times_\kk S_n$ (\cite[Lemma~2.24]{KumashiroMiyashita2025}). Away from the graded maximal ideals, localizations on the $i$-th branch agree with those of $S_i$. Thus each $S_i$ is regular on its punctured spectrum, and \autoref{prop:sr-regular-punctured-simplex} makes $\Delta_i$ a simplex. Equidimensionality of $S$ forces their dimensions to agree.

Conversely, (3) makes $S$ a fiber product of polynomial rings of the same positive dimension. \autoref{thm:fiber-top-trace-exact-formula} proves (1) and the stated value of the differential trace. For the canonical trace, the case $n=1$ is immediate, and the case $d\ge2$ follows from \cite[Theorem~3.20(2)]{KumashiroMiyashita2025}. If $d=1$, then $S$ is the Cohen--Macaulay ring of $n$ coordinate axes: it is Gorenstein for $n\le2$, and has canonical trace $\mm_S$ for $n\ge3$ by \cite[Theorem~A]{MiyashitaVarbaro2025}.
\end{proof}

\appendix

\section{An alternative proof of \autoref{thm:LocalEquivRegular} via Nash transforms}\label{sect4}

We give a geometric proof of \autoref{thm:LocalEquivRegular} under equidimensionality, using Nobile's theorem on Nash transforms \cite{nobile1975Properties}.
Throughout this section,
let \(\kk\) be a field of characteristic \(0\) and let \((A, \mfrakm)\) be a reduced equidimensional local \(\kk\)-algebra essentially of finite type.
We set $d=\delta_{\kk}(A)=\dim A+\trdeg_{\kk}\kappa(A)$.


\begin{theorem}\label{TraceOmegaRegular}
    \(A\) is regular if and only if \(\tr_A(\Omega^d_{A/\kk}) = A\).
\end{theorem}

Before proving this theorem, we first prepare some results on sheaves of differentials.

\begin{remark}
The statement of \autoref{TraceOmegaRegular} is translated as follows: Let \(\kk\) be a field of characteristic \(0\) and let \(X\) be a reduced scheme of finite type over \(\kk\).
Choose a point $x\in X$ such that $A=\mcalO_{X,x}$ is equidimensional.
Then the \(\kk\)-scheme \(X\) is smooth at \(x\) if and only if \(\tr_A(\Omega^d_{A/\kk}) = A\).
\end{remark}

\begin{lemma} \label{TraceTorsionFreeQuotient}
    Let \(R\) be a ring and let \(M\) be an \(R\)-module.
    Then the trace ideal \(\tr_R(M)\) of \(M\) is identified with the trace ideal \(\tr_R(M/M_{\tors})\) of the torsion-free quotient \(M/M_{\tors}\) of \(M\).
\end{lemma}

\begin{proof}
    By the definition of the trace ideal, it suffices to show that the natural injection
    \begin{equation*}
        \Hom_R(M/M_{\tors}, R) \to \Hom_R(M, R)
    \end{equation*}
    is an isomorphism.
    Let \(f \in \Hom_R(M, R)\) be a homomorphism and let \(m \in M_{\tors}\) be a torsion element of \(M\).
    Then there exists a non-zero-divisor \(r \in R\) such that \(rm = 0\) and thus \(rf(m) = f(rm) = 0\).
    Since \(r\) is a non-zero-divisor, we have \(f(m) = 0\) and thus \(f\) factors through the torsion-free quotient \(M/M_{\tors}\).
\end{proof}

\begin{lemma} \label{InvertibleTrace}
    Let \(X\) be a reduced equidimensional scheme of finite type over a field \(\kk\) of characteristic \(0\) and let \(\mcalM\) be a torsion-free coherent sheaf on \(X\) which is of rank \(1\) on a dense open subset of \(X\).
    Choose a point \(x \in X\).
    The stalk \(\mcalM_x\) of \(\mcalM\) at \(x\) satisfies \(\tr_{\mcalO_{X, x}}(\mcalM_x) = \mcalO_{X, x}\) if and only if \(\mcalM\) is invertible on an open neighborhood of \(x\).
\end{lemma}
\begin{proof}
    The `if' part is clear because the trace ideal of an invertible module is the whole ring.

    We show the `only if' part.
    Set \(A \defeq  \mcalO_{X, x}\).
    Since \(A\) is local, the equality \(\tr_{A}(\mcalM_x) = A\) implies that there exists a surjective \(A\)-linear homomorphism \(\varphi \colon \mcalM_x \to A\).

    Since \(\mcalM\) is torsion-free and of rank \(1\) on a dense open subset of \(X\), the torsion-free \(A\)-module \(\mcalM_x\) is of rank \(1\) at any minimal prime ideal of \(A\).
    Therefore, we have an isomorphism \(\mcalM_x \otimes_A Q(A) \cong Q(A)\) of \(Q(A)\)-modules, where \(Q(A)\) is the total quotient ring of \(A\).
    Taking the tensor product of the exact sequence
    \begin{equation*}
        0 \to \Ker(\varphi) \to \mcalM_x \xrightarrow{\varphi} A \to 0
    \end{equation*}
    with \(Q(A)\) over \(A\), we obtain a surjective homomorphism \(Q(A) \cong \mcalM_x \otimes_A Q(A) \to Q(A)\) which implies that this is an isomorphism and thus \(\Ker(\varphi) \otimes_A Q(A) = 0\).
    Since \(\Ker(\varphi)\) is a finitely generated \(A\)-module, \(\Ker(\varphi) \subseteq \mcalM_x\) is torsion but \(\mcalM_x\) is torsion-free and thus \(\Ker(\varphi) = 0\).
    This implies that \(\mcalM_x \cong A\).

    Since \(\mcalM\) is coherent, the locus where \(\mcalM\) is locally free of rank \(1\) is open in \(X\) and contains \(x\).
    So there exists an open neighborhood \(U\) of \(x\) such that \(\restr{\mcalM}{U}\) is a line bundle on \(U\).
\end{proof}

\begin{lemma} \label{TorsionFreeQuotient}
    Let \(\kk\) be a field of characteristic \(0\) and let \(X\) be a reduced equidimensional scheme of finite type over \(\kk\) of pure dimension \(d\).
    Let \(T \subseteq \Omega^d_{X/\kk}\) be the torsion subsheaf of \(\Omega^d_{X/\kk}\).
    Then the torsion-free quotient \(\Omega_{X/\kk}^d/T\) is of rank \(1\) on a dense open subset of \(X\).
\end{lemma}

\begin{proof}
    Since \(X\) is reduced and equidimensional of pure dimension \(d\), the smooth locus \(X_{\sm}\) of \(X\) is open and dense in \(X\) and the restriction \(\restr{\Omega^d_{X/\kk}}{X_{\sm}}\) is a line bundle on \(X_{\sm}\) (\citeSta{02G1}).
    In particular, the restriction \(\restr{T}{X_{\sm}}\) of \(T\) is zero and thus the restriction \(\restr{\Omega^d_{X/\kk}/T}{X_{\sm}}\) of \(\Omega^d_{X/\kk}/T\) is a line bundle on \(X_{\sm}\).
    This implies that the torsion-free sheaf \(\Omega^d_{X/\kk}/T\) is of rank \(1\) on a dense open subset of \(X\).
\end{proof}

We now prove the main result of this section.

\begin{proof}[Proof of \autoref{TraceOmegaRegular}]
The case $d=0$ is immediate. If $A$ is regular, $\Omega^1_{A/\kk}$ is free of rank $d$ and the trace ideal of the top exterior power is $A$. Conversely, write $A=\mcalO_{X,x}$ with $X$ affine of finite type over $\kk$.
Then every irreducible component of $X$ containing $x$ has dimension $d$. Removing the other irreducible components and shrinking further, we may assume that $X$ is reduced and all irreducible components have dimension $d$ and pass through $x$.

Set $\mcalM=\Omega^d_{X/\kk}/\mathrm{tors}$. By \autoref{TraceTorsionFreeQuotient}, \autoref{InvertibleTrace}, and \autoref{TorsionFreeQuotient}, the hypothesis \(\tr_A(\Omega^d_{A/\kk}) = A\) makes $\mcalM$ invertible near $x$. Shrink once more and embed $X$ into a smooth affine $\kk$-scheme $V$. Set $E=\Omega^1_{V/\kk}|_X$. The locally free quotient $\bigwedge^d E\twoheadrightarrow\mcalM$ of rank \(1\) defines a section
\begin{equation}\label{MorphismsToProjectiveBundle}
r \colon X \to \setP_X(\bigwedge^d E) = \Grass^1(\bigwedge^d E)
\end{equation}
of the projection \(\setP_X(\bigwedge^d E) \to X\).
On the smooth locus \(X_{\sm}\), the quotient \(\restr{E}{X_{\sm}} \twoheadrightarrow \restr{\Omega^1_{X/\kk}}{X_{\sm}}\) is a locally free quotient of rank \(d\) and defines a morphism
\begin{equation}\label{SurjectionOnSmoothLocus}
\gamma \colon X_{\sm} \to \Grass_X^d(E)
\end{equation}
of \(X\)-schemes.
On $X_{\sm}$, the quotient $\bigwedge^d E\twoheadrightarrow\mcalM$ is the $d$-th exterior power of the quotient $\restr{E}{X_{\sm}}\twoheadrightarrow\Omega^1_{X_{\sm}/\kk}$. Thus the section $\restr{r}{X_{\sm}}$ decomposes as
\begin{equation*}
    \restr{r}{X_{\sm}} \colon X_{\sm} \xrightarrow{\gamma} \Grass^d(E) \xrightarrow{\omega^d_E} \Grass^1(\bigwedge^d E) = \setP_X(\bigwedge^d E),
\end{equation*}
where \(\omega^d_E\) is the Pl\"ucker embedding induced from the exterior power \(\bigwedge^d\) of vector bundles, which is a closed immersion (\cite{gortz2020Algebraic}*{Remark 8.24}).

Choose the image \(\mcalI\) in \(\mcalO_X\) of the pullback along \(r\) of the ideal sheaf of \(\setP_X(\bigwedge^d E)\) corresponding to the closed immersion \(\omega^d_E\).
Then \(\mcalI \subseteq \mcalO_X\) restricts to the zero ideal on the dense open subset \(X_{\sm}\) because \(\restr{r}{X_{\sm}}\) factors through the closed immersion \(\omega^d_E\) as above.
Then any section of \(\mcalI\) is zero on any generic point of \(X\) and the reducedness of \(X\) shows \(\mcalI \cong 0\).
Thus $r$ factors through a section $\widetilde{\gamma}:X\to \Grass_X^d(E)$ of the projection \(\Grass_X^d(E) \to X\).
Since the projection $\Grass_X^d(E)\to X$ is separated, the section $\widetilde{\gamma}$ is a closed immersion. Moreover, $\restr{\widetilde{\gamma}}{X_{\sm}}=\gamma$, so its image is the closure of $\gamma(X_{\sm})$. Hence the Nash transform \(\widehat{X} \defeq  \overline{\gamma(X_{\sm})}\) is isomorphic to \(\widetilde{\gamma}(X)\), which is isomorphic to \(X\) via the projection to \(X\) since \(\widetilde{\gamma}\) is a splitting.

After base change to an algebraic closure \(\overline{\kk}\) of $\kk$, the scheme \(X_{\overline{\kk}}\) remains reduced and pure-dimensional, its smooth locus remains dense, and the same section identifies its Nash transform with itself. Nobile's theorem (\cite[Theorem~2]{nobile1975Properties}) therefore makes \(X_{\kk}\) smooth. Smoothness descends along the field extension, so $X$ is smooth at $x$ and $A$ is regular.
\end{proof}


\begin{bibdiv}
\begin{biblist}*{labels={alphabetic}}

\bib{aoyama1983some}{article}{
  author={Aoyama, Y{\^o}ichi},
  title={Some basic results on canonical modules},
  date={1983},
  journal={Journal of Mathematics of Kyoto University},
  volume={23},
  number={1},
  pages={85--94},
}

\bib{aoyama1985endomorphism}{article}{
  author={Aoyama, Y{\^o}ichi},
  author={Goto, Shiro},
  title={On the endomorphism ring of the canonical module},
  date={1985},
  journal={Journal of Mathematics of Kyoto University},
  volume={25},
  number={1},
  pages={21--30},
}

\bib{BH}{book}{
  author={Bruns, Winfried},
  author={Herzog, J\"{u}rgen},
  title={Cohen--Macaulay rings},
  series={Cambridge Studies in Advanced Mathematics},
  publisher={Cambridge University Press, Cambridge},
  date={1993},
  volume={39},
  isbn={0-521-41068-1},
}

\bib{celikbas2023traces}{misc}{
  author={Celikbas, Ela},
  author={Herzog, J\"{u}rgen},
  author={Kumashiro, Shinya},
  title={Traces of semi-invariants},
  date={2023},
  note={arXiv:2312.00983},
  url={https://arxiv.org/abs/2312.00983},
}

\bib{costa1979Polynomial}{article}{
  author={Costa, Douglas~L.},
  title={The polynomial rank of a commutative ring},
  date={1979},
  journal={Communications in Algebra},
  volume={7},
  number={14},
  pages={1509--1530},
}

\bib{DaoKobayashiTakahashi2021}{article}{
  author={Dao, Hailong},
  author={Kobayashi, Toshinori},
  author={Takahashi, Ryo},
  title={Trace ideals of canonical modules, annihilators of {Ext} modules, and classes of rings close to being {Gorenstein}},
  date={2021},
  journal={Journal of Pure and Applied Algebra},
  volume={225},
  number={9},
  pages={106655},
}

\bib{E}{book}{
  author={Eisenbud, David},
  title={Commutative algebra},
  series={Graduate Texts in Mathematics},
  publisher={Springer-Verlag, New York},
  date={1995},
  volume={150},
  isbn={0-387-94268-8},
  url={https://doi.org/10.1007/978-1-4612-5350-1},
  note={With a view toward algebraic geometry},
}

\bib{Ficarra2025}{article}{
  author={Ficarra, Antonino},
  title={The canonical trace of {Cohen--Macaulay} algebras of codimension $2$},
  date={2025},
  journal={Proceedings of the American Mathematical Society},
  volume={153},
  number={8},
  pages={3275--3289},
}

\bib{FicarraHerzogStamateTrivedi2024}{article}{
  author={Ficarra, Antonino},
  author={Herzog, J\"{u}rgen},
  author={Stamate, Dumitru~I.},
  author={Trivedi, Vijaylaxmi},
  title={The canonical trace of determinantal rings},
  date={2024},
  journal={Archiv der Mathematik},
  volume={123},
  number={5},
  pages={487--497},
}

\bib{gortz2020Algebraic}{book}{
  author={G\"ortz, Ulrich},
  author={Wedhorn, Torsten},
  title={Algebraic geometry I: Schemes. With examples and exercises},
  edition={2},
  series={Springer Studium Mathematik -- Master},
  publisher={Springer Spektrum},
  date={2020},
  url={https://www.springer.com/gp/book/9783658307325},
}

\bib{gortz2023Algebraic}{book}{
  author={G\"ortz, Ulrich},
  author={Wedhorn, Torsten},
  title={Algebraic geometry II: Cohomology of schemes. With examples and exercises},
  series={Springer Studium Mathematik -- Master},
  publisher={Springer Fachmedien},
  date={2023},
  url={https://link.springer.com/10.1007/978-3-658-43031-3},
}

\bib{goto1978graded}{article}{
  author={Goto, Shiro},
  author={Watanabe, Keiichi},
  title={On graded rings, I},
  date={1978},
  journal={Journal of the Mathematical Society of Japan},
  volume={30},
  number={2},
  pages={179--213},
}

\bib{graf2015generalized}{article}{
  author={Graf, Patrick},
  title={The generalized {Lipman--Zariski} problem},
  date={2015},
  journal={Mathematische Annalen},
  volume={362},
  number={1--2},
  pages={241--264},
  doi={10.1007/s00208-014-1112-9},
}

\bib{M2}{misc}{
  author={Grayson, Daniel~R.},
  author={Stillman, Michael~E.},
  title={Macaulay2, a software system for research in algebraic geometry},
  note={Available at \url{http://www2.macaulay2.com}},
}

\bib{hartshorne2013algebraic}{book}{
  author={Hartshorne, Robin},
  title={Algebraic geometry},
  publisher={Springer Science \& Business Media},
  date={2013},
}

\bib{herzog2019trace}{article}{
  author={Herzog, J\"{u}rgen},
  author={Hibi, Takayuki},
  author={Stamate, Dumitru~I.},
  title={The trace of the canonical module},
  date={2019},
  journal={Israel Journal of Mathematics},
  volume={233},
  number={1},
  pages={133--165},
}



\bib{iyengar2016annihilation}{article}{
  author={Iyengar, Srikanth~B.},
  author={Takahashi, Ryo},
  title={Annihilation of cohomology and strong generation of module categories},
  date={2016},
  journal={International Mathematics Research Notices},
  volume={2016},
  number={2},
  pages={499--535},
}

\bib{Kimura2026Schubert}{misc}{
  author={Kimura, Kaito},
  title={Trace ideals of canonical modules over {Schubert} cycles and determinantal rings},
  date={2026},
  note={arXiv:2601.18387},
  url={https://arxiv.org/abs/2601.18387},
}

\bib{KumashiroMiyashita2025}{misc}{
  author={Kumashiro, Shinya},
  author={Miyashita, Sora},
  title={Canonical traces of graded fiber products: applications to disconnected {Stanley--Reisner} rings},
  date={2025},
  note={arXiv:2506.04899v5},
  url={https://arxiv.org/abs/2506.04899v5},
}

\bib{lescot2006serie}{incollection}{
  author={Lescot, Jack},
  title={La s{\'e}rie de Bass d'un produit fibr{\'e} d'anneaux locaux},
  date={1983},
  booktitle={S{\'e}minaire d'alg{\`e}bre Paul Dubreil et Marie-Paule Malliavin: Proceedings, Paris 1982 (35{\`e}me ann{\'e}e)},
  publisher={Springer},
  pages={218--239},
}

\bib{lindo2017trace}{article}{
  author={Lindo, Haydee},
  title={Trace ideals and centers of endomorphism rings of modules over commutative rings},
  date={2017},
  journal={Journal of Algebra},
  volume={482},
  pages={102--130},
}

\bib{maitra2022partial}{article}{
  author={Maitra, Sarasij},
  title={Partial trace ideals and {Berger}'s conjecture},
  date={2022},
  journal={Journal of Algebra},
  volume={598},
  pages={1--23},
  doi={10.1016/j.jalgebra.2022.01.018},
}

\bib{miller2019co}{article}{
  author={Miller, Claudia},
  author={Vassiliadou, Sophia},
  title={(co)torsion of exterior powers of differentials over complete intersections},
  date={2019},
  journal={Journal of Singularities},
  volume={19},
  pages={131--162},
  doi={10.5427/jsing.2019.19h},
}

\bib{Miyashita2024LinearGeneralization}{article}{
  author={Miyashita, Sora},
  title={A linear generalization of the nearly {Gorenstein} property, with applications to {Veronese} subalgebras},
  date={2026},
  journal={Mathematische Nachrichten},
  volume={299},
  number={9},
  pages={2436--2451},
  doi={10.1002/mana.70199},
}

\bib{MiyashitaVarbaro2025}{article}{
  author={Miyashita, Sora},
  author={Varbaro, Matteo},
  title={The canonical trace of {Stanley--Reisner} rings that are {Gorenstein} on the punctured spectrum},
  date={2025},
  journal={International Mathematics Research Notices},
  number={12},
  pages={rnaf176},
}

\bib{nobile1975Properties}{article}{
  author={Nobile, A.},
  title={Some properties of the {Nash} blowing-up},
  date={1975},
  journal={Pacific Journal of Mathematics},
  volume={60},
  number={1},
  pages={297--305},
  url={https://projecteuclid.org/journals/pacific-journal-of-mathematics/volume-60/issue-1/Some-properties-of-the-Nash-blowing-up/pjm/1102868640.full},
}

\bib{stacks-project}{misc}{
  label={Sta},
  author={{The Stacks Project Authors}},
  title={Stacks Project},
  url={https://stacks.math.columbia.edu},
}

\bib{vasconcelos1991computing}{article}{
  author={Vasconcelos, Wolmer~V.},
  title={Computing the integral closure of an affine domain},
  date={1991},
  journal={Proceedings of the American Mathematical Society},
  volume={113},
  number={3},
  pages={633--638},
}

\bib{wright1981jacobian}{article}{
  author={Wright, David},
  title={On the Jacobian conjecture},
  date={1981},
  journal={Illinois Journal of Mathematics},
  volume={25},
  number={3},
  pages={423--440},
}

\bib{zariski1965Studies}{article}{
  author={Zariski, Oscar},
  title={Studies in equisingularity. I. Equivalent singularities of plane algebroid curves},
  date={1965},
  journal={American Journal of Mathematics},
  volume={87},
  number={2},
  pages={507--536},
  url={https://www.jstor.org/stable/2373019},
}

\end{biblist}
\end{bibdiv}
\end{document}